\documentclass[11pt, leqno, twoside]{article}

\usepackage{amsfonts}
\usepackage{amssymb}
\usepackage{amsmath}
\usepackage{amsthm}
\usepackage{xcolor}
\usepackage{mathrsfs}
\usepackage{txfonts}

\usepackage{indentfirst}

\allowdisplaybreaks

\def\red{\color{red}}

\def\rr{{\mathbb R}}
\def\rn{{{\rr}^n}}

\def\zz{{\mathbb Z}}

\def\nn{{\mathbb N}}

\def\ch{{\mathcal H}}

\def\cm{{\mathcal M}}

\def\cp{{\mathcal P}}

\def\cs{{\mathcal S}}

\def\cx{{\mathcal X}}

\def\RR{\mathbb R}

\def\fz{\infty}
\def\az{\alpha}

\def\ez{\epsilon}

\def\lz{\lambda}

\def\oz{\omega}

\def\tz{\theta}

\def\ch1{\mathbf{1}}
\def\lf{\left}
\def\r{\right}

\def\ls{\lesssim}

\def\noz{\nonumber}

\def\st{\subset}

\def\supp{\mathop\mathrm{\,supp\,}}

\def\lv{{L^{p(\cdot)}(\rn)}}

\newtheorem{theorem}{Theorem}[section]
\newtheorem{lemma}[theorem]{Lemma}

\newtheorem{proposition}[theorem]{Proposition}
\theoremstyle{definition}
\newtheorem{remark}[theorem]{Remark}
\newtheorem{definition}[theorem]{Definition}

\numberwithin{equation}{section}

\begin{document}

\arraycolsep=1pt

\title{\bf\Large The Variable Muckenhoupt Weight Revisited
\footnotetext {\hspace{-0.35cm}
2020 {\it Mathematics Subject Classification}.
Primary 42B25; Secondary 42B30, 42B20, 46E30.
\endgraf {\it Key words and phrases}.
variable Muckenhoupt weight, Hardy--Littlewood maximal operator,
weighted variable Hardy space, real-variable characterization,
Calder\'on--Zygmund operator, Bochener--Riesz means.}}

\author{\protect\parbox{\textwidth}{\protect\centering
Hongchao Jia and Xianjie Yan\footnote{Corresponding
author/{\red May 6, 2024}/Final version.}\\
\vspace{0.3cm}
\noindent\footnotesize{Institute of Contemporary Mathematics,
School of Mathematics and Statistics, Henan University,
Kaifeng 475004, The People's Republic of China\\
E-mails: \texttt{hongchaojia@henu.edu.cn} (H. Jia);\ 
\texttt{xianjieyan@henu.edu.cn} (X. Yan)}}}
\date{ }
\maketitle

\vspace{-0.8cm}

\begin{center}
\begin{minipage}{13cm}
{\small {\bf Abstract}\quad
Let $p(\cdot):\ \mathbb R^n\to(0,\infty)$ be a variable
exponent function and $X$ a ball quasi-Banach function space.
In this paper, we first study the relationship between two
kinds of variable weights $\mathcal{W}_{p(\cdot)}(\mathbb{R}^n)$
and $A_{p(\cdot)}(\mathbb{R}^n)$. Then, by regarding the weighted
variable Lebesgue space $L^{p(\cdot)}_{\omega}(\mathbb{R}^n)$
with $\omega\in\mathcal{W}_{p(\cdot)}(\mathbb{R}^n)$ as a special
case of $X$ and applying known results of the Hardy-type
space $H_{X}(\mathbb{R}^n)$ associated with $X$,
we further obtain several equivalent characterizations of
the weighted variable Hardy space $H^{p(\cdot)}_{\omega}(\rn)$
and the boundedness of some sublinear operators on $H^{p(\cdot)}_{\omega}(\rn)$.
All of these results coincide with or improve existing ones,
or are completely new.}
\end{minipage}
\end{center}


\section{Introduction}\label{s1}

Given a locally integrable function $f$ on $\rn$,
the Hardy--Littlewood maximal operator $\cm(f)$ is defined by setting,
for any $x\in\rn$,
\begin{align*}
{\mathcal M}(f)(x):=\sup_{B\ni x}\frac1{|B|}\int_B |f(y)|\,dy.
\end{align*}
It is well-known that the boundedness of $\cm$ on weighted Lebesgue spaces
$L^{p}(\omega,\rn)$, with $p\in[1,\fz)$, is closely related to classical
Muckenhoupt weights (see, for instance, \cite{cf74,mu72}).
Based on this, the real-variable theory,
including some equivalent characterizations and boundedness of operators,
of weighted Hardy spaces has been developed
(see \cite{gr85,st}). On the other hand,
since intrinsic properties of variable function
spaces are widely applied to harmonic analysis,
partial differential equations and variational integrals with
nonstandard growth conditions (see \cite{dr03,fan08,mi06}),
real-variable theory of them has attracted many attentions
in recent years (see, for instance, \cite{ah10,cw14,ns12,xu08}).
Recall that the variable Lebesgue space $L^{p(\cdot)}(\rn)$
with $p(\cdot):\ \rn\rightarrow(0,\fz)$ is defined to be the
set of all measurable functions $f$ such that
\begin{equation*}
\|f\|_{\lv}:=\inf\lf\{\lz\in(0,\fz):\
\int_\rn\lf[\frac{|f(x)|}{\lz}\r]^{p(x)}\,dx\le1\r\}<\fz.
\end{equation*}
The boundedness of $\cm$ on $L^{p(\cdot)}(\rn)$
has also been studied in many academic literature
(see, for instance, \cite{cfbook,dhr11}).

There is an increasing interest to explore what conditions are
necessary for the boundedness of $\cm$ on the weighted variable
Lebesgue space because it generalizes both $L^{p}(\omega,\rn)$
and $L^{p(\cdot)}(\rn)$. To this end, lots of efforts are made;
see \cite{c22,c24,cc22,ghb18,kok21,kos08}.
Especially, Diening and H\"ast\"o \cite{dh-pre} first introduced
the variable Muckenhoupt weight $\mathbb{A}_{p(\cdot)}(\rn)$
to characterize the boundedness of $\cm$ on weighted variable
Lebesgue spaces $L^{p(\cdot)}(\omega,\rn)$, where the class
$\mathbb{A}_{p(\cdot)}(\rn)$ is defined to consist of all
weights $\omega$ such that, for any ball $B\st\rn$,
\begin{align*}
\lf\|\omega^{1/p(\cdot)}\ch1_{B}\r\|_{L^{p(\cdot)}(\rn)}
\lf\|\omega^{-1/p(\cdot)}\ch1_{B}\r\|_{L^{p'(\cdot)}(\rn)}\ls|B|,
\end{align*}
$p'(\cdot)$ denotes the conjugate variable exponent of $p(\cdot)$,
and $\omega\,dx$ is treated as a measure, namely,
\begin{equation*}
\|f\|_{L^{p(\cdot)}(\omega,\rn)}:=\lf\|f\omega^{1/p(\cdot)}\r\|_{L^{p(\cdot)}(\rn)}.
\end{equation*}
For more studies about $\mathbb{A}_{p(\cdot)}(\rn)$, we refer
the reader to \cite{ak21,ins15,inns20,inns23,mz16,ns21}.
Later, Cruz-Uribe et al. \cite{cdh11,cfn12} introduced other
variable Muckenhoupt weight $A_{p(\cdot)}(\rn)$ to characterize
the boundedness of $\cm$ on weighted variable Lebesgue spaces
$L^{p(\cdot)}_{\omega}(\rn)$, where the class
$A_{p(\cdot)}(\rn)$ is defined to consist of all
weights $\omega$ such that, for any ball $B\st\rn$,
\begin{align*}
\lf\|\omega\ch1_{B}\r\|_{L^{p(\cdot)}(\rn)}
\lf\|\omega\ch1_{B}\r\|_{L^{p'(\cdot)}(\rn)}\ls|B|,
\end{align*}
and $\omega$ is regarded as a multiplier, namely,
\begin{equation*}
\|f\|_{L^{p(\cdot)}_{\omega}(\rn)}:=\lf\|f\omega\r\|_{L^{p(\cdot)}(\rn)}.
\end{equation*}
For more studies about $A_{p(\cdot)}(\rn)$, we refer
the reader to \cite{bbd21,cdmo20,cg20,cw17,ho21,tan23}.
However, the real-variable theory of the weighted variable Hardy space $H_{\omega}^{p(\cdot)}(\mathbb{R}^n)$
with $\omega\in A_{p(\cdot)}(\rn)$ is still unknown.

To present some light on the above question,
Ho \cite{ho17} introduced the variable weight
$\mathcal{W}_{p(\cdot)}(\rn)$ in order to establish the
atomic characterizations of
$H^{p(\cdot)}_{\omega}(\rn)$ and obtain some important operators on it.
Since then, various variants of $H^{p(\cdot)}_{\omega}(\rn)$ have been
introduced and their real-variable theories have been well developed.
To be precise, Laadjal et al. \cite{lsmm22} established
molecular characterizations and the duality of
$H^{p(\cdot)}_{L,\omega}(\rn)$ with operators $L$
satisfying Davies--Gaffney estimates;
Melkemi et al. \cite{msm21} studied weighted variable
Hardy spaces $H^{p(\cdot)}_{\omega}$ on domains.
For more studies about $\mathcal{W}_{p(\cdot)}(\rn)$, we refer
the reader to \cite{ho19,r23}.

On the other hand, Sawano et al. \cite{shyy17} first
introduced the ball quasi-Banach function spaces $X$,
which extend quasi-Banach function spaces in \cite{cfbook} further
so that weighted Lebesgue spaces and Morrey spaces are included.
Sawano et al. in \cite{shyy17} also introduced the Hardy-type spaces ${H_X({\RR}^n)}$,
associated with $X$, and established their various equivalent
characterizations in terms of maximal functions, atoms,
molecules, and Lusin area functions. Later,
Wang et al. \cite{wyy} established Littlewood--Paley
functions characterizations of ${H_X({\RR}^n)}$ and
obtained the boundedness of Calder\'on--Zygmund
operators on ${H_X({\RR}^n)}$. Furthermore,
Yan et al. \cite{yyy20} established intrinsic square
function characterizations of ${H_X({\RR}^n)}$.
For more studies about the real-variable theory of ${H_X({\RR}^n)}$,
we refer the reader to \cite{cwyz19,ho23,tz23,wyy23}.

Very recently, Sawano points out that it is worth comparing
the variable weights $\mathcal{W}_{p(\cdot)}(\rn)$ and
$A_{p(\cdot)}(\rn)$ when he reviews \cite{lsmm22}.
Motivated by this and \cite{shyy17}, in this paper,
we first clarify the relation between variable weights
$\mathcal{W}_{p(\cdot)}(\rn)$ and $A_{p(\cdot)}(\rn)$.
Then, by viewing $L^{p(\cdot)}_{\omega}(\rn)$ with
$\omega\in\mathcal{W}_{p(\cdot)}(\rn)$ as a special case of
$X$ and applying known results of $H_X({\RR}^n)$,
we further establish several equivalent characterizations of
$H^{p(\cdot)}_{\omega}(\rn)$  in terms of maximal functions,
atoms, molecules, (intrinsic) Lusin area functions,
and (intrinsic) Littlewood--Paley $g$-functions and $g_{\lz}^*$-functions;
and also obtain the boundedness of some sublinear operators,
including Calder\'on--Zygmund operators and Bochner--Riesz means,
on $H^{p(\cdot)}_{\omega}(\rn)$. It is remarkable to point out that,
one part of these results improves or coincides with existing ones
and the other part is completely new.

The organization of the remainder of this article is as follows.

In Section \ref{s2}, we aim to clarify the relation between
$\mathcal{W}_{p(\cdot)}(\rn)$ and $A_{p(\cdot)}(\rn)$.
To this end, in Subsection \ref{s2x}, we first recall definitions
of $\mathcal{W}_{p(\cdot)}(\rn)$ and $L^{p(\cdot)}_{\omega}(\rn)$.
In Subsection \ref{s2y}, we introduce the concept of ball
quasi-Banach function spaces $X$ and show that
$L^{p(\cdot)}_{\omega}(\rn)$ with $\omega\in\mathcal{W}_{p(\cdot)}(\rn)$
is a special case of $X$. In Subsection \ref{s2z},
we first recall the definition of $A_{p(\cdot)}(\rn)$.
Then, by using the two key lemmas about $X$ (see Lemmas \ref{embedlem1}
and \ref{embedlem2} below), we show the relations between
$\mathcal{W}_{p(\cdot)}(\rn)$ and $A_{p(\cdot)}(\rn)$
(see Propositions \ref{relation1} and \ref{relation3} below).

Section \ref{s3} is devoted to characterizing the space
$H^{p(\cdot)}_{\omega}(\rn)$ in terms of maximal functions,
atoms, molecules, (intrinsic) Lusin area functions,
and (intrinsic) Littlewood--Paley $g$-functions and $g_{\lz}^*$-functions.
To this end, we first introduce the weighted variable Hardy space
$H^{p(\cdot)}_{\omega}(\rn)$ via Lusin area functions
in Subsection \ref{s3x}. Then, in Subsection \ref{s3y},
we prove that $L^{p(\cdot)}_{\omega}(\rn)$ with
$\omega\in\mathcal{W}_{p(\cdot)}(\rn)$ supports a
Fefferman--Stein vector-valued maximal inequality,
and that the powered Hardy--Littlewood maximal operator
is bounded on the associate space of $L^{p(\cdot)}_{\omega}(\rn)$
(see Theorem \ref{condition} below).
Finally, by viewing $H_\omega^{p(\cdot)}({{\RR}^n})$ as special
cases of Hardy spaces $H_X({{\RR}^n})$ associated with ball
quasi-Banach function spaces $X$ and applying known results of $H_X({{\RR}^n})$,
we obtain equivalent characterizations of $H_\omega^{p(\cdot)}({{\RR}^n})$
in Subsection \ref{s3z}.

In Section \ref{s4}, we establish the boundedness of Calder\'on--Zygmund
operators on $H_\omega^{p(\cdot)}({{\RR}^n})$ and of Bochner--Riesz means
from $H_\omega^{p(\cdot)}({{\RR}^n})$ to $L_\omega^{p(\cdot)}({{\RR}^n})$,
respectively, in Theorems \ref{cz-bdn} and \ref{br-bdn}.
Especially, the obtained results improves the known ones.

At the end of this section,
we make some conventions on notation.
Let $\nn:=\{1,2,\ldots\}$ and $\zz_+:=\nn\cup\{0\}$.
We denote by $C$ a \emph{positive constant} which is independent
of the main parameters, but may vary from line to line.
We use $C_{(\az,\dots)}$ to denote a positive constant depending
on the indicated parameters $\az,\, \dots$.
The symbol $f\ls g$ means $f\le Cg$
and, if $f\ls g\ls f$, then we write $f\sim g$.
If $f\le Cg$ and $g=h$ or $g\le h$,
we then write $f\ls g=h$ or $f\ls g\le h$,
rather than $f\ls g\sim h$ or $f\ls g\ls h$.
If $E$ is a subset of $\cx$, we denote by ${\mathbf{1}}_E$ its
\emph{characteristic function} and by $E^\complement$
the set $\cx\setminus E$.
For any $x\in\rn$ and $r\in(0,\fz)$, we denote by $B(x,r)$ the ball
centered at $x$ with the radius $r$, namely,
$B(x,r):=\{y\in\rn:\ d(x,y)<r\}.$
For any ball $B$, we use $x_B$ to denote its center and $r_B$ its radius,
and denote by $\lz B$ for any $\lz\in(0,\fz)$ the ball concentric with
$B$ having the radius $\lz r_B$. For any $\az\in\rr$,
we denote by $\lfloor\az\rfloor$ the largest integer
not greater than $\az$. For any index $q\in[1,\fz]$,
we denote by $q'$ its \emph{conjugate index}, namely, $1/q+1/q'=1$.

\section{The Relation Between Two Variable Muckenhoupt Weights\label{s2}}

In this section, we first recall some basic concepts about
Hardy--Littlewood maximal operators and weighted
variable Lebesgue spaces in Subsection \ref{s2x}.
We also recall some basic concepts about ball
quasi-Banach function spaces in Subsection \ref{s2y},
which are used throughout this article.
Then we introduce two variable Muckenhoupt weights
and clarify their relation in Subsection \ref{s2z}.

\subsection{Weighted Variable Lebesgue Spaces\label{s2x}}

In this subsection, we recall some concepts of
weighted variable Lebesgue spaces.

A measurable function $p(\cdot):\ \rn\to(0,\fz)$ is called a \emph{variable exponent}.
Denote by $\cp(\rn)$ the \emph{set of all variable exponents}
$p(\cdot)$ satisfying
\begin{align}\label{2.1x}
0<p_-:=\mathop\mathrm{ess\,inf}_{x\in \rn}p(x)\le
\mathop\mathrm{ess\,sup}_{x\in \rn}p(x)=:p_+<\fz.
\end{align}

In what follows, for any $p(\cdot)\in\cp(\rn)$,
let
\begin{align}\label{2.28.x1}
p_*:=\min\{p_-,1\}
\end{align}
and $p'(\cdot)$ be the \emph{conjugate variable exponent} of $p(\cdot)$,
namely, for any $x\in\rn$, $\frac{1}{p(x)}+\frac{1}{p'(x)}=1$.

For any given measurable function $f$ on $\rn$ and $p(\cdot)\in\cp(\rn)$,
the \emph{modular functional} (or, simply, the \emph{modular})
$\varrho_{p(\cdot)}$, associated with $p(\cdot)$, is defined by
setting $$\varrho_{p(\cdot)}(f):=\int_\rn|f(x)|^{p(x)}\,dx$$
and the \emph{Luxemburg} (also known as \emph{Luxemburg--Nakano})
\emph{quasi-norm} is given by setting
\begin{equation*}
\|f\|_{\lv}:=\inf\lf\{\lz\in(0,\fz):\
\varrho_{p(\cdot)}\lf(\frac{f}{\lz}\r)\le1\r\}.
\end{equation*}
Then the \emph{variable Lebesgue space} $\lv$ is defined to be the
set of all measurable functions $f$ such that $\varrho_{p(\cdot)}(f)<\fz$,
equipped with the quasi-norm $\|f\|_{\lv}$.

Now, we recall the definition of weighted variable
Lebesgue spaces from \cite[Definition 2.2]{ho17}
(see also \cite[p.\,70]{lsmm22}).

\begin{definition}\label{wvls}
Let $p(\cdot)\in\cp(\rn)$ and $\omega$ be a measurable function
such that $\omega\in(0,\fz)$ almost everywhere in $\rn$.
The \emph{weighted variable Lebesgue space} $L^{p(\cdot)}_{\omega}(\rn)$
is defined to be the set of all measurable functions on ${{\RR}^n}$
such that
\begin{align*}
\|f\|_{L^{p(\cdot)}_{\omega}(\rn)}:=\|f\omega\|_{L^{p(\cdot)}(\rn)}<\fz.
\end{align*}
\end{definition}

For any given $r\in(0,{\infty})$, we denote by $L_{\rm loc}^r({{\RR}^n})$
the \emph{set of all $r$-locally integrable functions on ${{\RR}^n}$}.
Recall that the \emph{Hardy--Littlewood maximal operator} ${\mathcal M}$
is defined by setting, for any $f\in L_{\rm loc}^1({{\RR}^n})$ and $x\in{{\RR}^n}$,
\begin{align*}
{\mathcal M}(f)(x):=\sup_{B\ni x}\frac1{|B|}\int_B |f(y)|\,dy,
\end{align*}
where the supremum is taken over all balls $B$ of ${{\RR}^n}$ containing $x$.

Next, we recall the definition of variable Muckenhoupt weights
from \cite[Definition 2.3]{ho17}
(see also \cite[Definition 1]{lsmm22}).

\begin{definition}\label{variable weight}
Let $p(\cdot)\in\cp(\rn)$.
We denote by $\mathcal{W}_{p(\cdot)}(\rn)$ the
\emph{set of all measurable functions $\omega$ on $\rn$}
such that $\omega\in(0,\fz)$ almost everywhere and
\begin{enumerate}
\item[{\rm(i)}] for any ball $B\subset\rn$,
$\|\ch1_B\|_{L^{p(\cdot)/p_*}_{\omega^{p_*}}(\rn)}<\fz$ and
$\|\ch1_B\|_{L^{(p(\cdot)/p_*)'}_{\omega^{-p_*}}(\rn)}<\fz$,
where $p_*$ is the same as in \eqref{2.28.x1};

\item[{\rm(ii)}] there exist $\kappa,s\in(1,\fz)$ such
that the Hardy--Littlewood maximal operator $\cm$ is bounded
on $L^{(sp(\cdot))'/\kappa}_{\omega^{-\kappa/s}}(\rn)$.
\end{enumerate}
\end{definition}

Let $p\in[1,\fz)$ and $\omega\in L^1_{\rm loc}(\rn)$ be a nonnegative function.
Recall that $\omega$ is said to be a \emph{Muckenhoupt $A_p$-weight},
denoted by $\omega\in A_p(\rn)$, if
\begin{align}\label{21.7.14.x1}
[\omega]_{A_p(\rn)}:=\sup_{B\subset\rn}\frac{1}{|B|^p}\int_B\omega(x)\,dx
\lf\{\int_B\lf[\omega(y)\r]^{1-p'}\,dy\r\}^{p/p'}<\fz
\end{align}
when $p\in(1,\fz)$, or
\begin{align*}
[\omega]_{A_1(\rn)}:=\sup_{B\subset\rn}\frac{1}{|B|}\int_B\omega(x)\,dx
\lf(\mathop\mathrm{ess\,sup}_{y\in B}\lf[\omega(y)\r]^{-1}\r)<\fz,
\end{align*}
where the suprema are taken over all balls
$B\subset\rn$ (see, for instance, \cite{st}).
Throughout this article, we always let
$$A_{\fz}(\rn):=\bigcup_{p\in[1,\fz)}A_{p}(\rn).$$

\begin{remark}\label{vwrem}
As was pointed out in \cite[p.\,388]{ho17},
when $p(\cdot)\equiv p\in(0,1]$,
Definition \ref{variable weight}(i) is equivalent to the assumption
that $\omega^p$ and $\omega^{-p'}$ are locally integrable;
when $p(\cdot)\equiv p\in(1,\fz)$,
Definition \ref{variable weight}(i) is equivalent to
the assumption that $\omega$ is locally integrable
and $\omega^{-1}$ is locally bounded;
when $p(\cdot)\equiv p\in(0,\fz)$,
Definition \ref{variable weight}(ii) is equivalent to
the assumption that $\omega^p\in A_{\fz}(\rn)$.
\end{remark}

\subsection{Ball Quasi-Banach Function Spaces}\label{s2y}

In this subsection, we recall some concepts of ball quasi-Banach function spaces.
To this end, we first recall the definition of ball quasi-Banach
function spaces (see, for instance, \cite[Definition 2.2]{shyy17}).
In what follows, we always let $\mathscr M(\rn)$ be the
\emph{set of all measurable functions} on ${{\RR}^n}$.

\begin{definition}\label{BQBFS}
A quasi-normed linear space $X\subset\mathscr M(\rn)$,
equipped with a quasi-norm $\|\cdot\|$ which
makes sense for the whole $\mathscr M(\rn)$, is called
a \emph{ball quasi-Banach function space} if it satisfies that
\begin{enumerate}
\item[{\rm(i)}] for any $f\in\mathscr M(\rn)$, $\|f\|_X=0$
implies that $f=0$ almost everywhere;

\item[{\rm(ii)}] for any $f,g\in\mathscr M(\rn)$, $|g|\le|f|$ almost everywhere
implies that $\|g\|_X\le\|f\|_X$;

\item[{\rm(iii)}] for any $\{f_m\}_{m\in\nn}\subset\mathscr M(\rn)$
and $f\in\mathscr M(\rn)$, $0\le f_m\uparrow f$ as $m\to\infty$
almost everywhere implies that
$\|f_m\|_X\uparrow\|f\|_X$ as $m\to\infty$;

\item[{\rm(iv)}] $\ch1_B \in X$ for any ball $B\subset\rn$.
\end{enumerate}

Moreover, a {ball quasi-Banach function} space $X$ is
called a \emph{ball Banach function space} if it satisfies:
\begin{enumerate}
\item[{\rm(v)}] for any $f,g\in X$, $\|f+g\|_X\leq\|f\|_X+\|g\|_X$;

\item[{\rm(vi)}] for any given ball $B\subset\rn$,
there exists a positive constant $C_{(B)}$ such that,
for any $f\in X$,
\begin{equation*}
\int_B |f(x)|\,dx \leq C_{(B)}\|f\|_X.
\end{equation*}
\end{enumerate}
\end{definition}

The associate space $X'$ of any given ball
Banach function space $X$ is defined as follows;
see \cite[Chapter 1, Section 2]{bs88} or \cite[p.\,9]{shyy17}.

\begin{definition}\label{associate space}
For any given ball Banach function space $X$,
its \emph{associate space} (also called the
\emph{K\"othe dual space}) $X'$ is defined by setting
\begin{equation*}
X':=\lf\{f\in\mathscr M(\rn):\ \|f\|_{X'}
:=\sup_{g\in X,\ \|g\|_X=1}\lf\|fg\r\|_{L^1(\rn)}<\infty\r\},
\end{equation*}
where $\|\cdot\|_{X'}$ is called the \emph{associate norm} of $\|\cdot\|_X$.
\end{definition}

Now, we recall the concept of the
$p$-convexification of ball quasi-Banach function spaces,
which is a part of \cite[Definition 2.6]{shyy17}.

\begin{definition}\label{Debf}
Let $X$ be a ball quasi-Banach function space and $p\in(0,\infty)$.
The \emph{$p$-convexification} $X^p$ of $X$ is defined by setting
$$X^p:=\lf\{f\in\mathscr M(\rn):\ |f|^p\in X\r\}$$
equipped with the \emph{quasi-norm} $\|f\|_{X^p}:=\|\,|f|^p\|_X^{1/p}$
for any $f\in X^p$.
\end{definition}

Next, we prove that the weighted variable Lebesgue space
$L^{p(\cdot)}_{\omega}(\rn)$, with $p(\cdot)\in\cp(\rn)$
and $\omega\in\mathcal{W}_{p(\cdot)}(\rn)$,
is a ball quasi-Banach function space.

\begin{lemma}\label{lx.x}
Let $p(\cdot)\in\cp(\rn)$ and $\omega\in\mathcal{W}_{p(\cdot)}(\rn)$.
Then $L^{p(\cdot)}_{\omega}(\rn)$ is a ball quasi-Banach function space.
Moreover, if $p_-\ge 1$, then $L^{p(\cdot)}_{\omega}(\rn)$
is a ball Banach function space.
\end{lemma}

\begin{proof}
On the one hand, by Definition \ref{wvls} and
\cite[Theorems 2.17 and 2.59]{cfbook}, we find that
$L^{p(\cdot)}_{\omega}(\rn)$ satisfies (i), (ii), and (iii)
of Definition \ref{BQBFS}. On the other hand,
from Definition \ref{variable weight}(i),
we deduce that, for any ball $B\st\rn$,
\begin{align*}
\|\ch1_B\|_{L^{p(\cdot)}_{\omega}(\rn)}^{p_*}
&=\|\omega\ch1_B\|_{L^{p(\cdot)}(\rn)}^{p_*}
=\lf[\inf\lf\{\lz\in(0,\fz):\ \varrho_{p(\cdot)}\lf(\frac{\omega\ch1_B}{\lz}\r)\le1\r\}\r]^{p_*}\\
&=\inf\lf\{\lz^{p_*}\in(0,\fz):\ \int_{B}\lf[\frac{\omega(x)}{\lz}\r]^{p(x)}\,dx\le1\r\}\\
&=\inf\lf\{\lz^{p_*}\in(0,\fz):\ \int_{B} \lf\{\frac{[\omega(x)]^{p_*}}{\lz^{p_*}}\r\}^{p(x)/p_*}\,dx\le1\r\}\\
&=\inf\lf\{t\in(0,\fz):\
\varrho_{p(\cdot)/p_*}\lf(\frac{\omega^{p_*}\ch1_B}{t}\r)\le1\r\}\\
&=\lf\|\omega^{p_*}\ch1_B\r\|_{L^{p(\cdot)/p_*}(\rn)}
=\lf\|\ch1_B\r\|_{L^{p(\cdot)/p_*}_{\omega^{p_*}}(\rn)}<\fz,
\end{align*}
which further implies that $L^{p(\cdot)}_{\omega}(\rn)$
satisfies Definition \ref{BQBFS}(iv).
Thus, $L^{p(\cdot)}_{\omega}(\rn)$ is a ball quasi-Banach function space.
If $p_-\ge 1$, then $p_*=1$. By this, Definition \ref{wvls},
\cite[Theorems 2.17 and 2.26]{cfbook},
and Definition \ref{variable weight}(i),
we conclude that $L^{p(\cdot)}_{\omega}(\rn)$ satisfies
Definition \ref{BQBFS}(v) and, for any given ball $B\st\rn$
and any $f\in L^{p(\cdot)}_{\oz}(\rn)$,
\begin{align*}
\int_B |f(x)|\,dx
&=\int_{\rn} |f(x)|\omega(x)\ch1_B(x)[\omega(x)]^{-1}\,dx
\ls\lf\|f\oz\r\|_{L^{p(\cdot)}(\rn)}
\lf\|\ch1_B\omega^{-1}\r\|_{L^{p'(\cdot)}(\rn)}\\
&=\lf\|f\r\|_{L^{p(\cdot)}_{\oz}(\rn)}
\lf\|\ch1_B\r\|_{L^{p'(\cdot)}_{\omega^{-1}}(\rn)}
=\lf\|f\r\|_{L^{p(\cdot)}_{\oz}(\rn)}
\|\ch1_B\|_{L^{(p(\cdot)/p_*)'}_{\omega^{-p_*}}(\rn)}.
\end{align*}
This implies that $L^{p(\cdot)}_{\omega}(\rn)$
satisfies Definition \ref{BQBFS}(vi) and hence
completes the proof of Lemma \ref{lx.x}.
\end{proof}

Now, we state a fundamental conclusion about the $\theta$-convexification
($\tz\in(0,\fz)$) of weighted variable Lebesgue spaces.

\begin{lemma}\label{lx.y}
Let $p(\cdot)\in\cp(\rn)$, $\omega\in\mathcal{W}_{p(\cdot)}(\rn)$, and $\tz\in(0,\fz)$.
Then $$\lf(L^{p(\cdot)}_{\omega}(\rn)\r)^{\tz}=L^{\tz p(\cdot)}_{\omega^{1/\tz}}(\rn)$$
with equivalent quasi-norms, where $(L^{p(\cdot)}_{\omega}(\rn))^{\tz}$
denotes the $\theta$-convexification of $L^{p(\cdot)}_{\omega}(\rn)$.
\end{lemma}

\begin{proof}
Indeed, by Lemma \ref{lx.x} and Definitions \ref{Debf} and \ref{wvls},
we find that, for any $f\in\mathscr M(\rn)$,
\begin{align*}
\|f\|_{(L^{p(\cdot)}_{\omega}(\rn))^{\tz}}
&=\lf\||f|^{\tz}\r\|_{L^{p(\cdot)}_{\omega}(\rn)}^{1/\tz}
=\lf\|\omega|f|^{\tz}\r\|_{L^{p(\cdot)}(\rn)}^{1/\tz}\\
&=\lf[\inf\lf\{\lz\in(0,\fz):\ \varrho_{p(\cdot)}\lf(\frac{\omega|f|^{\tz}}{\lz}\r)\le1\r\}\r]^{1/\tz}\\
&=\inf\lf\{\lz^{1/\tz}\in(0,\fz):\ \int_{\rn}\lf[\frac{\omega(x)|f(x)|^{\tz}}{\lz}\r]^{p(x)}\,dx\le1\r\}\\
&=\inf\lf\{\lz^{1/\tz}\in(0,\fz):\
\int_{B}\lf\{\frac{[\omega(x)]^{1/\tz}|f(x)|}{\lz^{1/\tz}}\r\}^{\tz p(x)}\,dx\le1\r\}\\
&=\inf\lf\{t\in(0,\fz):\
\varrho_{\tz p(\cdot)}\lf(\frac{\omega^{1/\tz}f}{t}\r)\le1\r\}\\
&=\lf\|\omega^{1/\tz}f\r\|_{L^{\tz p(\cdot)}(\rn)}
=\lf\|f\r\|_{L^{\tz p(\cdot)}_{\omega^{1/\tz}}(\rn)},
\end{align*}
which completes the proof of Lemma \ref{lx.y}.
\end{proof}

At the end of this subsection, we state a fundamental
conclusion about associate spaces of weighted variable Lebesgue spaces.

\begin{lemma}\label{1.17.x1}
Let $p(\cdot)\in\cp(\rn)$ satisfy $p_-\in[1,\fz)$
with $p_-$ the same as in \eqref{2.1x},
and $\omega\in\mathcal{W}_{p(\cdot)}(\rn)$.
Then $$\lf(L^{p(\cdot)}_{\omega}(\rn)\r)'=L^{p'(\cdot)}_{\omega^{-1}}(\rn)$$
with equivalent quasi-norms, where $p'(\cdot)$ is the
conjugate variable exponent of $p(\cdot)$.
\end{lemma}

\begin{proof}
Indeed, by Lemma \ref{lx.x}, Definition \ref{associate space},
and \cite[Proposition 2.2]{ho17}, we find that,
for any $f\in(L^{p(\cdot)}_{\omega}(\rn))'$,
\begin{align*}
\|f\|_{(L^{p(\cdot)}_{\omega}(\rn))'}
&=\sup_{g\in L^{p(\cdot)}_{\omega}(\rn),\
\|g\|_{L^{p(\cdot)}_{\omega}(\rn)}=1}\lf\|fg\r\|_{L^1(\rn)}
\sim\|f\|_{L^{p'(\cdot)}_{\omega^{-1}}(\rn)},
\end{align*}
where the implicit positive equivalence constant is independent of $f$.
\end{proof}

\subsection{Relations between Two Variable Muckenhoupt Weights\label{s2z}}

In this subsection, we introduce another definition
of the variable Muckenhoupt weight and clarify
the relationship between it and the weight in
Definition \ref{variable weight}. To this end,
we first recall the concept of the variable
Muckenhoupt weight from \cite[p.\,364]{cdh11}
(see also \cite[Definition 1.4]{cfn12}).

\begin{definition}\label{variable weight2}
Let $p(\cdot)\in\cp(\rn)$ satisfy $p_-\in[1,\fz)$ with
$p_-$ the same as in \eqref{2.1x} and $\omega$ be
a measurable function such that $\omega\in(0,\fz)$
almost everywhere in $\rn$. We say $\omega\in A_{p(\cdot)}(\rn)$
if there exists a positive constant $C$ such that,
for any ball $B\subset\rn$,
\begin{align}\label{2.1x3}
\lf\|\omega\ch1_{B}\r\|_{L^{p(\cdot)}(\rn)}
\lf\|\omega^{-1}\ch1_{B}\r\|_{L^{p'(\cdot)}(\rn)}\le C|B|.
\end{align}
\end{definition}

\begin{remark}\label{weight2rem}
\begin{enumerate}
\item[(i)] By Definition \ref{variable weight2},
we conclude that $\omega\in A_{p(\cdot)}(\rn)$ if and only if
$\omega^{-1}\in A_{p'(\cdot)}(\rn)$.

\item[(ii)]
If $p(\cdot)\equiv p\in[1,\fz)$,
then $\omega\in A_{p(\cdot)}(\rn)$ if and only if $\omega^p\in A_p(\rn)$.
This conclusion can be found in \cite[p.\,746]{cfn12}
(see also \cite[p.\,365]{cdh11}). But we still give some
details of its proof here for the sake of completeness.
Let $p(\cdot)\equiv p\in(1,\fz)$. By \eqref{2.1x3} and \eqref{21.7.14.x1},
we conclude that
\begin{align*}
\omega\in A_{p(\cdot)}(\rn)
&\Longleftrightarrow\sup_{B\subset\rn}\frac{1}{|B|}
\lf\|\omega\ch1_{B}\r\|_{L^{p}(\rn)}\lf\|\omega^{-1}\ch1_{B}\r\|_{L^{p'}(\rn)}\le C\\
&\Longleftrightarrow\sup_{B\subset\rn}\frac{1}{|B|}
\lf\{\int_B\lf[\omega(x)\r]^p\,dx\r\}^{1/p}
\lf\{\int_B\lf[\lf(\omega(y)\r)^p\r]^{1-p'}\,dy\r\}^{1/p'}\le C\\
&\Longleftrightarrow\sup_{B\subset\rn}\frac{1}{|B|^p}
\int_B\lf[\omega(x)\r]^p\,dx
\lf\{\int_B\lf[\lf(\omega(y)\r)^p\r]^{1-p'}\,dy\r\}^{p/p'}\le C^p\\
&\Longleftrightarrow\omega^p\in A_p(\rn).
\end{align*}
Some usual modifications are made when $p(\cdot)\equiv p=1$.
\end{enumerate}
\end{remark}

Via Definition \ref{variable weight2} and arguments similar to those
used in the proofs of Lemmas \ref{lx.x}, \ref{lx.y}, and \ref{1.17.x1},
we have the following result; and we omit the details of its proofs.

\begin{lemma}\label{lx.xx}
Let $p(\cdot)\in\cp(\rn)$ satisfy $p_-\in[1,\fz)$
with $p_-$ the same as in \eqref{2.1x}.
Then the conclusions of Lemmas \ref{lx.x}, \ref{lx.y},
and \ref{1.17.x1} still hold true with
$\omega\in\mathcal{W}_{p(\cdot)}(\rn)$
replaced by $\omega\in A_{p(\cdot)}(\rn)$.
\end{lemma}

A function $p(\cdot)\in\cp(\rn)$ is said to satisfy the
\emph{globally log-H\"older continuous condition}, denoted by
$p(\cdot)\in C^{\log}(\rn)$, if there exist positive constants
$C_{\log}(p)$ and $C_\fz$, and $p_\fz\in\rr$ such that,
for any $x,y\in\rn$,
\begin{align*}
\lf|p(x)-p(y)\r|\le \frac{C_{\log}(p)}{\log(e+1/|x-y|)}
\end{align*}
and
\begin{align*}
\lf|p(x)-p_\fz\r|\le \frac{C_\fz}{\log(e+|x|)}.
\end{align*}

The following conclusion is the main result of this subsection.

\begin{proposition}\label{relation1}
Let $p(\cdot)\in\cp(\rn)$ satisfy $p_-\in(1,\fz)$ with $p_-$
the same as in \eqref{2.1x}, and $\omega$ be a measurable
function such that $\omega\in(0,\fz)$ almost everywhere in $\rn$.
If $p(\cdot)\in C^{\log}(\rn)$ and $\omega\in A_{p(\cdot)}(\rn)$,
then $\omega\in\mathcal{W}_{p(\cdot)}(\rn)$.
\end{proposition}

To prove Proposition \ref{relation1}, we need more preparations.
For any given $\theta\in(0,{\infty})$, the \emph{powered Hardy--Littlewood
maximal operator} ${\mathcal M}^{({\theta})}$ is defined by setting, for any
$f\in L_{\rm loc}^1({{\RR}^n})$ and $x\in{{\RR}^n}$,
\begin{align*}
{\mathcal M}^{({\theta})}(f)(x):=\left\{{\mathcal M}\left(|f|^{{\theta}}\right)(x)\right\}^{\frac{1}{{\theta}}}.
\end{align*}
The following conclusion comes from \cite[Lemma 2.15(ii)]{shyy17}.

\begin{lemma}\label{embedlem1}
Assume that $X$ is a ball quasi-Banach function space on $\rn$
and $\cm$ bounded on $X$. Then there exists an $\eta\in(1,\fz)$
such that $\cm^{(\eta)}$ is bounded on $X$.
\end{lemma}

The following lemma is just \cite[Lemma 2.9]{cwyz19}.

\begin{lemma}\label{embedlem2}
Let $X$ be a ball Banach function space and $\eta\in[1,\fz)$.
If $\cm$ is bounded on $X$, then $\cm$ is also bounded on $X^{\eta}$.
\end{lemma}

The following conclusion was obtained in \cite[Theorem 1.5]{cfn12}
(see also \cite[Theorem 1.3]{cdh11}).

\begin{lemma}\label{embedlem3}
Let $p(\cdot)\in C^{\log}(\rn)$ satisfy $p_-\in(1,\fz)$
with $p_-$ the same as in \eqref{2.1x}.
Then $\cm$ is bounded on $L^{p(\cdot)}_{\omega}(\rn)$
if and only if $\omega\in A_{p(\cdot)}(\rn)$.
\end{lemma}

\begin{remark}\label{lem3rem}
Let $p(\cdot)\in C^{\log}(\rn)$ satisfy $p_-\in(1,\fz)$
with $p_-$ the same as in \eqref{2.1x}.
By Remark \ref{weight2rem}(i) and \cite[Remark 4.1.5]{dhr11},
we find that $\cm$ is bounded on $L^{p(\cdot)}_{\omega}(\rn)$
if and only if $\cm$ is bounded on
$L^{p'(\cdot)}_{\omega^{-1}}(\rn)$.
\end{remark}

Next, we prove Proposition \ref{relation1}.

\begin{proof}[Proof of Proposition \ref{relation1}]
On the one hand, since $p_-\in(1,\fz)$, it follows that $p_*=1$.
Moreover, by $\omega\in A_{p(\cdot)}(\rn)$ and Definition \ref{variable weight2},
we find that, for any $B\in\rn$,
$$\|\ch1_B\|_{L^{p(\cdot)/p_*}_{\omega^{p_*}}(\rn)}
=\|\ch1_B\|_{L^{p(\cdot)}_{\omega}(\rn)}
<\fz\quad\mbox{and}\quad
\|\ch1_B\|_{L^{(p(\cdot)/p_*)'}_{\omega^{-p_*}}(\rn)}
=\|\ch1_B\|_{L^{p'(\cdot)}_{\omega^{-1}}(\rn)}<\fz,$$
which implies Definition \ref{variable weight}(i).

On the other hand, from $p(\cdot)\in C^{\log}(\rn)$,
$\omega\in A_{p(\cdot)}(\rn)$, and Lemma \ref{embedlem3},
we deduce that $\cm$ is bounded on $L^{p(\cdot)}_{\omega}(\rn)$,
which, combined with Lemmas \ref{lx.xx} and \ref{embedlem2},
implies that there exists a constant $\eta_1\in(1,\fz)$
such that $\cm$ is bounded on
$(L^{p(\cdot)}_{\omega}(\rn))^{\eta_1}
=L^{\eta_1p(\cdot)}_{\omega^{1/\eta_1}}(\rn)$.
This, together with Remark \ref{lem3rem},
further shows that $\cm$ is bounded on
$L^{(\eta_1p(\cdot))'}_{\omega^{-1/\eta_1}}(\rn)$.
By this and Lemma \ref{embedlem1},
we conclude that there exists a constant $\eta_2\in(1,\fz)$
such that $\cm^{(\eta_2)}$ is bounded on $L^{(\eta_1p(\cdot))'}_{\omega^{-1/\eta_1}}(\rn)$
and hence, for any
$f\in L^{(\eta_1p(\cdot))'/\eta_2}_{\omega^{-\eta_2/\eta_1}}(\rn)$,
\begin{align*}
\lf\|\cm(f)\r\|_{L^{(\eta_1p(\cdot))'/\eta_2}_{\omega^{-\eta_2/\eta_1}}(\rn)}
&=\lf\|\cm(f)\r\|_{(L^{(\eta_1p(\cdot))'}_{\omega^{-1/\eta_1}}(\rn))^{1/\eta_2}}
=\lf\|\lf[\cm(f)\r]^{1/\eta_2}\r\|
_{L^{(\eta_1p(\cdot))'}_{\omega^{-1/\eta_1}}(\rn)}^{\eta_2}\\
&=\lf\|\cm^{(\eta_2)}\lf(|f|^{1/\eta_2}\r)\r\|
_{L^{(\eta_1p(\cdot))'}_{\omega^{-1/\eta_1}}(\rn)}^{\eta_2}
\ls\lf\||f|^{1/\eta_2}\r\|
_{L^{(\eta_1p(\cdot))'}_{\omega^{-1/\eta_1}}(\rn)}^{\eta_2}\\
&=\|f\|_{(L^{(\eta_1p(\cdot))'}_{\omega^{-1/\eta_1}}(\rn))^{1/\eta_2}}
=\|f\|_{L^{(\eta_1p(\cdot))'/\eta_2}_{\omega^{-\eta_2/\eta_1}}(\rn)},
\end{align*}
which implies Definition \ref{variable weight}(ii) with
$\kappa$ and $s$ therein replaced, respectively, by $\eta_2$ and $\eta_1$ here.
This completes the proof of Proposition \ref{relation1}.
\end{proof}

\begin{remark}\label{relation1rem}
\begin{enumerate}
\item[(i)] We point out that the proof of Proposition \ref{relation1}
depends heavily on the boundedness of $\cm$ on $L^{p(\cdot)}_{\omega}(\rn)$.
Hence $p(\cdot)\in C^{\log}(\rn)$ is not necessary in Proposition \ref{relation1}.
Indeed, $p(\cdot)\in C^{\log}(\rn)$ is not necessary for the boundedness
of $\cm$ on $L^{p(\cdot)}_{\omega}(\rn)$.
Fix $p_\fz\in(1,\fz)$ and define
\begin{align*}
\phi(x):=
\begin{cases}
\frac1k-|e^{k^2}-x|\ \ &\text{if}\ 0\le|e^{k^2}-x|\le\frac1k,\ 1\le k<\fz,\\
0\ \ &\text{otherwise}.
\end{cases}
\end{align*}
For any given $x\in\rn$, let $p(x):=\phi(x)+p_\fz$ and $\omega(x)\equiv1$.
Then, by \cite[Proposition 4.9, Definition 4.45 and Theorem 4.52]{cfbook},
we conclude that $p(\cdot)\notin C^{\log}(\rn)$ and $\cm$ is bounded on $L^{p(\cdot)}_{\omega}(\rn)$. From this and an argument similar to that
used in the proof of Proposition \ref{relation1},
we deduce that $\omega\in\mathcal{W}_{p(\cdot)}(\rn)$.
Recall that Cruz-Uribe and Wang in \cite[Remark 2.5]{cw17} conjectured that,
if $\cm$ is bounded on $L^{p(\cdot)}(\rn)$ and $\omega\in A_{p(\cdot)}(\rn)$,
then $\cm$ is bounded on $L^{p(\cdot)}_{\omega}(\rn)$,
which is still open. Therefore, what conditions can be weakened
in Proposition \ref{relation1} remains unknown.

\item[(ii)]
We initially wanted to prove the \emph{conclusion} that:
if $p(\cdot)\in C^{\log}(\rn)$ and
$\omega^{p_*}\in A_{p(\cdot)/p_*}(\rn)$
with $p_*$ the same as in \eqref{2.28.x1},
then $\omega\in\mathcal{W}_{p(\cdot)}(\rn)$.
This generalizes Proposition \ref{relation1}
by removing the limitation $p_-\in(1,\fz)$.
Indeed, when $p(\cdot)\equiv p\in(0,\fz)$,
the conclusion is easily proved.
Recall that, for any given $p\in(0,\fz)$ and $\omega\in A_{\fz}(\rn)$,
the \emph{weighted Lebesgue space $L^{p}(\omega,\rn)$} is defined to
be the set of all measurable functions $f$ on $\rn$ such that
$$\|f\|_{L^{p}(\omega,\rn)}:=
\lf\{\int_{\rn}|f(x)|^p\omega(x)\,dx\r\}^{1/p}<\infty.$$
Let $p(\cdot)\equiv p\in(0,\fz)$ and $\omega$ be a measurable
function such that $\omega\in(0,\fz)$ almost everywhere in $\rn$.
If $p\in[1,\fz)$, then, by Remarks \ref{weight2rem}(ii)
and \ref{vwrem}, we conclude that
\begin{align*}
\omega^{p_*}\in A_{p(\cdot)/p_*}(\rn)&\Longleftrightarrow
\omega^p\in A_p(\rn)\Longrightarrow\omega^p\in A_{\fz}(\rn)
\Longleftrightarrow\omega\in\mathcal{W}_{p(\cdot)}(\rn).
\end{align*}
If $p\in(0,1)$, then, from Remark \ref{weight2rem}(ii),
\cite[Propositions 7.2 and 7.6(1)]{duo}, and \cite[Theorem 7.3]{duo},
we deduce that there exist $\kappa\in(1,\fz)$ and $s\in(1/p,\fz)$ such that
\begin{align*}
\omega^{p_*}\in A_{p(\cdot)/p_*}(\rn)
&\Longleftrightarrow\omega^p\in A_1(\rn)\Longrightarrow\omega^p\in A_{sp}(\rn)\\
&\Longleftrightarrow\omega^{p[1-(sp)']}\in A_{(sp)'}(\rn)\noz\\
&\Longrightarrow\omega^{p[1-(sp)']}\in A_{(sp)'/\kappa}(\rn)\noz\\
&\Longleftrightarrow\cm\ \mbox{is bounded on}\
L^{(sp)'/\kappa}(\omega^{p[1-(sp)']},\rn)\noz\\
&\Longleftrightarrow\cm\ \mbox{is bounded on}\
L^{(sp)'/\kappa}(\omega^{-\kappa/s\cdot(sp)'/\kappa},\rn)\noz\\
&\Longleftrightarrow\cm\ \mbox{is bounded on}\
L^{(sp(\cdot))'/\kappa}_{\omega^{-\kappa/s}}(\rn)\noz\\
&\Longleftrightarrow\omega\in\mathcal{W}_{p(\cdot)}(\rn).\noz
\end{align*}
Recall that Cruz-Uribe and Wang in \cite[Theorem 2.4]{cw17} proved that the
boundedness of $\cm$ on $L^{p(\cdot)}_{\omega}(\rn)$ implies $p_->1$,
which makes it impossible to prove the conclusion with an argument
similar to the proof of Proposition \ref{relation1}.
Thus, whether the conclusion is right or not is still unknown.

\item[(iii)]
Denote by $\mathcal{A}_{p(\cdot)}(\rn)$ the set of all measurable functions
$\omega\in(0,\fz)$ almost everywhere in $\rn$ satisfying that,
for any families $\mathfrak{Q}$ of disjoint cubes and any
$f\in L^{p(\cdot)}_{\omega}(\rn)$,
\begin{align*}
\lf\|\sum_{Q\in\mathfrak{Q}}\ch1_Q\frac{1}{|Q|}\int_Q|f(y)|\,dy\r\|
_{L^{p(\cdot)}_{\omega}(\rn)}
\le C\|f\|_{L^{p(\cdot)}_{\omega}(\rn)},
\end{align*}
where $C$ is a positive constant independent of $f$.
Recall that Ho \cite[Proposition 3.7]{ho19} proved that,
if $p(\cdot)\in C^{\log}(\rn)$ and
$\omega^{p_*}\in\mathcal{A}_{p(\cdot)/p_*}(\rn)$,
then $\omega\in\mathcal{W}_{p(\cdot)}(\rn)$.
Indeed, by \cite[Theorem 4.5.7 and Remark 4.5.8]{dhr11}
(see also \cite[p.\,367]{cdh11}),
we conclude that $\omega\in\mathcal{A}_{p(\cdot)}(\rn)$
implies that $\omega\in A_{p(\cdot)}(\rn)$.
Thus, Proposition \ref{relation1} improves the corresponding
results in \cite[Proposition 3.7]{ho19} with $p_-\in(1,\fz)$
by weakening the condition $\omega\in\mathcal{A}_{p(\cdot)}(\rn)$
into $\omega\in A_{p(\cdot)}(\rn)$.

\item[(iv)]
Let $\omega\equiv1$.
If $p(\cdot)\in C^{\log}(\rn)$ satisfies $p_-\in(1,\fz)$
with $p_-$ the same as in \eqref{2.1x}, then,
by \cite[Remark 2.40, Proposition 2.3, and Theorem 3.16]{cfbook},
we conclude that, for any ball $B\subset\rn$,
$$\|\ch1_B\|_{L^{p(\cdot)/p_*}_{\omega^{p_*}}(\rn)}<\fz\quad
\mbox{and}\quad\|\ch1_B\|_{L^{(p(\cdot)/p_*)'}_{\omega^{-p_*}}(\rn)}<\fz,$$
and that there exist $s\in(1,\fz)$ and
$\kappa\in(1,(sp_+)')$ such that $\cm$ is bounded on
$L^{(sp(\cdot))'/\kappa}(\rn)$. This coincides with the
conclusion of Proposition \ref{relation1} with $\omega\equiv1$ therein.
\end{enumerate}
\end{remark}

\begin{proposition}\label{relation3}
Let $p(\cdot)\in\cp(\rn)$ satisfy $p_-\in[1,\fz)$.
If $\omega\in\mathcal{W}_{p(\cdot)}(\rn)$,
then there exists a constant $s\in(1,\fz)$
such that $\omega^{1/s}\in A_{sp(\cdot)}(\rn)$.
\end{proposition}

\begin{proof}
Indeed, by $\omega\in\mathcal{W}_{p(\cdot)}(\rn)$ and Remark \ref{vwrem},
we find that $\omega\in(0,\fz)$ almost everywhere in $\rn$
and there exist constants $\kappa,s\in(1,\fz)$ such that $\cm$
is bounded on $L^{(sp(\cdot))'/\kappa}_{\omega^{-\kappa/s}}(\rn)
=(L^{(sp(\cdot))'}_{\omega^{-1/s}}(\rn))^{1/\kappa}$.
From these and Lemma \ref{embedlem2}, we further deduce that
$\cm$ is bounded on $L^{(sp(\cdot))'}_{\omega^{-1/s}}(\rn)$.
By this, Lemma \ref{embedlem3}, and Remark \ref{weight2rem}(i),
we obtain $\omega^{1/s}\in A_{sp(\cdot)}(\rn)$.
\end{proof}

\section{Some Real-Variable Characterizations of
Weighted Variable Hardy Spaces\label{s3}}

In this section, we first introduce weighted variable Hardy spaces
$H_\omega^{p(\cdot)}({{\RR}^n})$ with $p(\cdot)\in\cp(\rn)$ and $\omega\in\mathcal{W}_{p(\cdot)}(\rn)$ in Subsection \ref{s3x}.
Then we verify that the weighted variable Lebesgue space
$L_\omega^{p(\cdot)}({{\RR}^n})$ supports a Fefferman--Stein
vector-valued maximal inequality and the boundedness of the
powered Hardy--Littlewood maximal operator on its associate space
in Subsection \ref{s3y}. Finally, in Subsection \ref{s3z},
we establish the real-variable characterizations of $H_\omega^{p(\cdot)}({{\RR}^n})$
in terms of maximal functions, atoms, molecules, (intrinsic) Lusin area functions,
and (intrinsic) Littlewood--Paley $g$-functions and $g_{\lz}^*$-functions
by viewing $H_\omega^{p(\cdot)}({{\RR}^n})$ as special cases
of Hardy spaces associated with ball quasi-Banach function spaces
$H_X({{\RR}^n})$ and applying known results of $H_X({{\RR}^n})$
obtained in \cite{cwyz19,shyy17,yyy20,yyy21,zhyy22}.

\subsection{Weighted Variable Hardy Spaces\label{s3x}}

To introduce the weighted variable Hardy space $H_\omega^{p(\cdot)}({{\RR}^n})$,
we first recall some basic concepts. Throughout this article,
we always let $\cs(\rn)$ the \emph{space of all Schwartz functions}
and $\cs'(\rn)$ its \emph{topological dual space} equipped with
the weak-$*$ topology, and let $\mathscr{F}$ and ${\mathscr{F}}^{-1}$
denote, respectively, the \emph{Fourier transform} and its \emph{inverse}.
Namely, for any $f\in{\mathcal S}({{\RR}^n})$ and $\xi\in{{\RR}^n}$,
\begin{align*}
\mathscr{F}f(\xi):=(2\pi)^{-\frac{n}{2}}\int_{{{\RR}^n}}f(x)e^{-ix\cdot\xi}\,dx
\quad\mbox{and}\quad{\mathscr{F}}^{-1}f(\xi):=\mathscr{F}f(-\xi),
\end{align*}
here and thereafter, for any $x:=(x_1,\dots,x_n)$ and
$\xi:=(\xi_1,\dots,\xi_n)\in{{\RR}^n}$,
$$x\cdot\xi:=\sum_{k=1}^nx_k\xi_k \quad\mbox{and}\quad i:=\sqrt{-1}.$$
For any $f\in{\mathcal S}'({{\RR}^n})$, $\mathscr{F}f$ is defined by setting,
for any $\varphi\in{\mathcal S}({{\RR}^n})$,
$\langle\mathscr{F}f,\varphi\rangle:=\langle f,\mathscr{F}\varphi\rangle$
and $\mathscr{F}^{-1}f$ is defined by setting, for any $\varphi\in{\mathcal S}({{\RR}^n})$, $\langle\mathscr{F}^{-1}f,\varphi\rangle:=\langle f,\mathscr{F}^{-1}\varphi\rangle$.

The following definition is just \cite[Definition 2.8]{yyy20}
(see also \cite[(3.13)]{shyy17} or \cite[Definition 2.8]{wyy}).

\begin{definition}\label{vz}
For any $t\in(0,{\infty})$, ${\varphi}\in{\mathcal S}({{\RR}^n})$,
and $f\in{\mathcal S}'({{\RR}^n})$, let
$${\varphi}(tD)(f):={\mathscr{F}}^{-1}\left({\varphi}(t\cdot)\mathscr{F}f\right).$$
\end{definition}

Let ${\varphi}\in{\mathcal S}({{\RR}^n})$. Recall that, for any
$f\in\mathcal{S}'(\mathbb R^n)$, the \emph{Lusin area function} $S(f)$
of $f$ is defined by setting, for any $x\in{{\RR}^n}$,
\begin{align*}
S(f)(x):=\left[\int_{\Gamma(x)}
\lf|{\varphi}(tD)(f)(y)\r|^2\,\frac{\,dy\,dt}{t^{n+1}}\right]^{1/2},
\end{align*}
here and thereafter, $\Gamma(x):=\{(y,t)\in\rn\times(0,\fz):\ |y-x|<t\}$.
Then we introduce the weighted variable Hardy space
$H_\omega^{p(\cdot)}({{\RR}^n})$,
which is a variant of \cite[Definition 2.4]{ho17}.

\begin{definition}\label{hardy}
Let $p(\cdot)\in\cp(\rn)$, $\omega\in\mathcal{W}_{p(\cdot)}(\rn)$,
and ${\varphi}\in{\mathcal S}({{\RR}^n})$ satisfy
$${\mathbf{1}}_{B(\vec{0}_n,4)\backslash B(\vec{0}_n,2)}\le{\varphi}
\le{\mathbf{1}}_{B(\vec{0}_n,8)\backslash B(\vec{0}_n,1)}.$$
Then the \emph{weighted variable Hardy space}
$H_\omega^{p(\cdot)}({{\RR}^n})$ is defined by setting
$$H_\omega^{p(\cdot)}({{\RR}^n}):=\left\{f\in{\mathcal S}'({{\RR}^n})
:\ \|f\|_{H_\omega^{p(\cdot)}({{\RR}^n})}
:=\left\|S(f)\right\|_{L_\omega^{p(\cdot)}({{\RR}^n})}<{\infty}\right\}.$$
\end{definition}

\begin{remark}\label{hardyrem}
If $p(\cdot)\equiv p\in(0,\fz)$ and $\omega\equiv1$, then,
$L_\omega^{p(\cdot)}({{\RR}^n})$ and $H_\omega^{p(\cdot)}({{\RR}^n})$
become, respectively, the classical Lebesgue space $L^p({{\RR}^n})$
and the classical Hardy space $H^p({{\RR}^n})$.
\end{remark}

\subsection{Two Important Conditions about
Hardy--Littlewood Maximal Operators\label{s3y}}

We first recall some basic concepts in this subsection.
The following definition can be found in \cite[p.\,387]{ho17}
(see also \cite[p.\,516]{r23} or \cite[p.\,71]{lsmm22}).

\begin{definition}\label{index}
Let $p(\cdot)\in\cp(\rn)$ and $\omega\in\mathcal{W}_{p(\cdot)}(\rn)$.
Define
\begin{align}\label{3.6.x1}
s_\omega:=\inf\left\{s\ge 1:\ \cm\ \mbox{is bounded on}\ L^{(sp(\cdot))'}_{\omega^{-1/s}}(\mathbb{R}^n)\right\}
\end{align}
and
\begin{align*}
\mathbb{S}_\omega:=\left\{s\ge 1:\ \cm\ \mbox{is bounded on}\ L^{(sp(\cdot))'/\kappa}_{\omega^{-\kappa/s}}(\mathbb{R}^n)\
\mbox{for some}\ \kappa>1\right\}.
\end{align*}
Moreover, for any given $s\in\mathbb{S}_\omega$, define
\begin{align}\label{3.6.x3}
\kappa_{\omega}^s:=\sup\left\{\kappa>1:\ \cm\ \mbox{is bounded on}\ L^{(sp(\cdot))'/\kappa}_{\omega^{-\kappa/s}}(\mathbb{R}^n)\right\}.
\end{align}
\end{definition}

\begin{remark}\label{inrem}
Let $p(\cdot)\in\cp(\rn)$ and $\omega\in\mathcal{W}_{p(\cdot)}(\rn)$.
\begin{enumerate}
\item[(i)] As was mentioned in \cite[p.\,388]{ho17},
the indices $s_\omega$ and $\kappa_{\omega}^s$ are defined for
presenting the atomic decomposition of $H_\omega^{p(\cdot)}({{\RR}^n})$.
Especially, $s_\omega$ is related to the vanishing moment
condition and $\kappa_{\omega}^s$ is related to the size
condition of the atomic decompositions,
which can also be seen in Theorem \ref{atomic} below.

\item[(ii)]
It is easy to see that $s_\omega\ge\frac{1}{p_*}$.
Moreover, if $s\in\mathbb{S}_\omega$, then $s\ge s_\omega$.
Indeed, by $s\in\mathbb{S}_\omega$, we conclude that there
exists some $\kappa\in(1,\fz)$ such that $\cm$ is bounded on $L^{(sp(\cdot))'/\kappa}_{\omega^{-\kappa/s}}(\mathbb{R}^n)$.
From this, \cite[(2.12)]{yhyy1}, Definition \ref{Debf},
and Lemmas \ref{lx.x} and \ref{lx.y},
we deduce that there exists a positive constant $C$ such that,
for any $f\in L^{(sp(\cdot))'}_{\omega^{-1/s}}(\mathbb{R}^n)$,
\begin{align*}
\left\|{\mathcal M}(f)\right\|
_{L^{(sp(\cdot))'}_{\omega^{-1/s}}(\mathbb{R}^n)}
&\le\left\|{\mathcal M}^{(\kappa)}(f)\right\|
_{L^{(sp(\cdot))'}_{\omega^{-1/s}}(\mathbb{R}^n)}
=\left\|\lf[{\mathcal M}\lf(\lf|f\r|^{\kappa}\r)\r]^{1/\kappa}\right\|
_{L^{(sp(\cdot))'}_{\omega^{-1/s}}(\mathbb{R}^n)}\\
&=\left\|{\mathcal M}\lf(\lf|f\r|^{\kappa}\r)\right\|
_{L^{(sp(\cdot))'/\kappa}_{\omega^{-\kappa/s}}(\mathbb{R}^n)}^{1/\kappa}
\le C\left\|\lf|f\r|^{\kappa}\right\|
_{L^{(sp(\cdot))'/\kappa}_{\omega^{-\kappa/s}}(\mathbb{R}^n)}^{1/\kappa}\\
&=C\left\|f\right\|_{L^{(sp(\cdot))'}_{\omega^{-1/s}}(\mathbb{R}^n)},
\end{align*}
which further implies that $\cm$ is bounded on
$L^{(sp(\cdot))'}_{\omega^{-1/s}}(\mathbb{R}^n)$.
Thus, $s\ge s_\omega$. This also shows that $s_\omega$
is well defined.

\item[(iii)]
By \eqref{3.6.x1}, we conclude that, for any given $s_0>s_\omega$,
there exists $s_1\in(s_\omega,s_0)$ such that $\cm$ is
bounded on $L^{(s_1p(\cdot))'}_{\omega^{-1/s_1}}(\mathbb{R}^n)$.

\item[(iv)]
Let $s\in\mathbb{S}_\omega$. It is easy to see that $\kappa_{\omega}^s$
is well defined. Moreover, from \eqref{3.6.x3},
we deduce that, for any given $\kappa_0<\kappa_{\omega}^s$,
there exists $\kappa_1\in(\kappa_0,\kappa_{\omega}^s)$ such that $\cm$ is
bounded on $L^{(sp(\cdot))'/\kappa_1}_{\omega^{-\kappa_1/s}}(\mathbb{R}^n)$.

\item[(v)]
If $p(\cdot)\equiv p\in(0,\fz)$ and $\omega\equiv1$, then,
by \cite[Theorem 2.5]{duo}, we obtain $s_{\omega}=\frac{1}{p_*}$
and $\kappa_{\omega}^{s_\omega}=(\max\{1,p\})'$.

\end{enumerate}
\end{remark}

The following conclusion shows the Fefferman--Stein vector-valued
maximal inequalities on $L^{p(\cdot)}_{\omega}(\mathbb{R}^n)$,
which is obtained in \cite[Theorem 3.1]{ho17}
(see also \cite[Theorem 2.4]{msm21} or \cite[Theorem 1]{lsmm22}).

\begin{lemma}\label{fsvector}
Let $p(\cdot)\in\cp(\rn)$ and $q\in(1,\fz)$.
If $\omega\in\mathcal{W}_{p(\cdot)}(\rn)$,
then there exists a positive constant $C$ such that,
for any sequences $\{f_j\}_{j=1}^\fz$ of measurable functions
and any $r\in(s_\omega,\fz)$,
\begin{align*}
\lf\|\lf\{\sum_{j=1}^\fz\lf[\cm(f_j)\r]^q\r\}^{1/q}\r\|
_{L^{rp(\cdot)}_{\omega^{1/r}}(\mathbb{R}^n)}
\le C\lf\|\lf(\sum_{j=1}^\fz|f_j|^q\r)^{1/q}\r\|
_{L^{rp(\cdot)}_{\omega^{1/r}}(\mathbb{R}^n)}.
\end{align*}
\end{lemma}

Now, we present two important conclusions about the boundedness
of the Hardy--Littlewood maximal operator on
$L^{p(\cdot)}_{\omega}(\mathbb{R}^n)$ and on its
associate space.

\begin{lemma}\label{condition1}
Let $p(\cdot)\in\cp(\rn)$ and $\omega\in\mathcal{W}_{p(\cdot)}(\rn)$.
Then there exist some ${\theta},h\in(0,1]$ with ${\theta}<h$,
and a positive constant $C$ such that,
for any sequences $\{f_j\}_{j=1}^\fz$ of measurable functions,
\begin{align*}
\left\|\left\{\sum_{j=1}^{{\infty}}\left[{\mathcal M}^{({\theta})}(f_j)\right]^h\right\}^{\frac1h}\right\|
_{L^{p(\cdot)}_{\omega}(\mathbb{R}^n)}
\le C\left\|\left\{\sum_{j=1}^{{\infty}}|f_j|^h\right\}^{\frac1h}\right\|
_{_{L^{p(\cdot)}_{\omega}(\mathbb{R}^n)}}.
\end{align*}
\end{lemma}

\begin{proof}
Indeed, choose $\tz\in(0,\frac{1}{s_\omega})$ and $h\in(\tz,1]$.
Then, by Lemmas \ref{lx.x} and \ref{lx.y},
Definition \ref{Debf}, and  Lemma \ref{fsvector},
we conclude that
\begin{align*}
\left\|\left\{\sum_{j=1}^{{\infty}}\left[{\mathcal M}^{({\theta})}(f_j)\right]^{h}\right\}^{\frac1h}\right\|
_{L^{p(\cdot)}_{\omega}(\mathbb{R}^n)}
&=\left\|\left\{\sum_{j=1}^{{\infty}}\left[{\mathcal M}\lf(\lf|f_j\r|^{\theta}\r)\right]^{\frac{h}{\tz}}\right\}^{\frac1h}\right\|
_{L^{p(\cdot)}_{\omega}(\mathbb{R}^n)}\\
&=\left\|\left\{\sum_{j=1}^{{\infty}}\left[{\mathcal M}\lf(\lf|f_j\r|^{\theta}\r)\right]^{\frac{h}{\tz}}\right\}^{\frac{\tz}{h}}\right\|
^{\frac{1}{\tz}}_{L^{p(\cdot)/\tz}_{\omega^\tz}(\mathbb{R}^n)}\\
&\ls\left\|\left\{\sum_{j=1}^{{\infty}}\lf|f_j\r|^{h}\right\}^{\frac{\tz}{h}}\right\|
^{\frac{1}{\tz}}_{L^{p(\cdot)/\tz}_{\omega^\tz}(\mathbb{R}^n)}
=\left\|\left\{\sum_{j=1}^{{\infty}}|f_j|^h\right\}^{\frac1h}\right\|
_{_{L^{p(\cdot)}_{\omega}(\mathbb{R}^n)}},
\end{align*}
where the implicit positive constant is independent of $\{f_j\}_{j=1}^\fz$.
This finishes the proof of Lemma \ref{condition1}.
\end{proof}

\begin{lemma}\label{condition2}
Let $p(\cdot)\in\cp(\rn)$ and $\omega\in\mathcal{W}_{p(\cdot)}(\rn)$.
Then there exist some $h\in(0,1]$, $q_0\in(1,\fz)$, and a positive
constant $C$ such that $(L^{p(\cdot)}_{\omega}(\mathbb{R}^n))^{1/h}$
is a ball Banach function space and,
for any $f\in((L^{p(\cdot)}_{\omega}(\mathbb{R}^n))^{1/h})'$,
\begin{align*}
\left\|{\mathcal M}^{((q_0/h)')}(f)\right\|
_{((L^{p(\cdot)}_{\omega}(\mathbb{R}^n))^{1/h})'}
\le C\|f\|_{((L^{p(\cdot)}_{\omega}(\mathbb{R}^n))^{1/h})'}.
\end{align*}
\end{lemma}

\begin{proof}
By $\omega\in\mathcal{W}_{p(\cdot)}(\rn)$,
we find that there exist constants $\kappa,s\in(1,\fz)$ such that
$\cm$ is bounded on $L^{(sp(\cdot))'/\kappa}_{\omega^{-\kappa/s}}(\rn)$.
Choose $h=\frac1s$ and $q_0\in(\max\{(\kappa_{\omega}^s)'/s,1\},\fz)$.
Then, from these and Remark \ref{inrem}(iv), we deduce that there exists
a constant $\kappa_1\in((q_0s)',\kappa_{\omega}^s)$ such that $\cm$ is
bounded on $L^{(sp(\cdot))'/\kappa_1}_{\omega^{-\kappa_1/s}}(\mathbb{R}^n)$,
which, combined with \cite[(2.12)]{yhyy1}, Lemma \ref{fsvector},
Definition \ref{Debf}, and Lemmas \ref{lx.x} and \ref{lx.y},
further implies that, for any $f\in((L^{p(\cdot)}_{\omega}(\mathbb{R}^n))^{1/h})'$,
\begin{align*}
\left\|{\mathcal M}^{((q_0/h)')}(f)\right\|
_{((L^{p(\cdot)}_{\omega}(\mathbb{R}^n))^{1/h})'}
&\le\left\|{\mathcal M}^{(\kappa_1)}(f)\right\|
_{((L^{p(\cdot)}_{\omega}(\mathbb{R}^n))^{1/h})'}
=\left\|\lf[{\mathcal M}\lf(\lf|f\r|^{\kappa_1}\r)\r]^{1/\kappa_1}\right\|
_{L^{(p(\cdot)/h)'}_{\omega^{-h}}(\mathbb{R}^n)}\\
&=\left\|{\mathcal M}\lf(\lf|f\r|^{\kappa_1}\r)\right\|^{1/\kappa_1}
_{L^{(p(\cdot)/h)'/{\kappa_1}}_{\omega^{-h\kappa_1}}(\mathbb{R}^n)}
=\left\|{\mathcal M}\lf(\lf|f\r|^{\kappa_1}\r)\right\|^{1/\kappa_1}
_{L^{(sp(\cdot))'/{\kappa_1}}_{\omega^{-\kappa_1/s}}(\mathbb{R}^n)}\\
&\le C\left\|\lf|f\r|^{\kappa_1}\right\|^{1/\kappa_1}
_{L^{(sp(\cdot))'/{\kappa_1}}_{\omega^{-\kappa_1/s}}(\mathbb{R}^n)}
=C\left\|f\right\|_{((L^{p(\cdot)}_{\omega}(\mathbb{R}^n))^{1/h})'},
\end{align*}
where $C$ is a positive constant independent of $f$.
Moreover, by Remark \ref{inrem}(ii) and Lemma \ref{lx.x},
we conclude that $h=\frac1s\le p_*$ and hence
$(L^{p(\cdot)}_{\omega}(\mathbb{R}^n))^{1/h}
=L^{p(\cdot)/h}_{\omega^h}(\mathbb{R}^n)$
is a ball Banach function space.
This finishes the proof of Lemma \ref{condition2}.
\end{proof}

Combining Lemmas \ref{condition1} and \ref{condition2},
we obtain the main result of this subsection as follows.

\begin{theorem}\label{condition}
Let $p(\cdot)\in\cp(\rn)$ and $\omega\in\mathcal{W}_{p(\cdot)}(\rn)$
with $s\in(1,\fz)$. Then $L^{p(\cdot)}_{\omega}(\mathbb{R}^n)$
satisfies all the conclusions of Lemmas \ref{condition1} and
\ref{condition2} with $\tz\in(0,\frac1s)$, $h=\frac1s$,
and $q_0\in(\max\{(\kappa_{\omega}^s)'/s,1\},\fz)$.
\end{theorem}

\begin{remark}\label{conrem}
If $p(\cdot)\equiv p\in(0,\fz)$ and $\omega\equiv1$, then,
by Remarks \ref{hardyrem} and \ref{inrem}(v), we conclude that
$L^{p}(\mathbb{R}^n)$ satisfies all the conclusions
of Lemmas \ref{condition1} and \ref{condition2} with
$\tz\in(0,p_*)$, $h=p_*$, and $q_0\in(\max\{1,p\},\fz)$,
which coincides with \cite[Remark 2.7(i)]{wyy}.
\end{remark}

\subsection{Several Real-Variable Characterizations of $H_\omega^{p(\cdot)}({{\RR}^n})$\label{s3z}}

In what follows, we always let
$H_{\omega,{\rm atom}}^{p(\cdot),r,d}(\rn)$,
$H_{\omega,{\rm fin}}^{p(\cdot),r,d}(\rn)$,
$H_{\omega,{\rm mol}}^{p(\cdot),r,d,\ez}(\rn)$
be, respectively, the weighted variable atomic Hardy space,
the weighted variable finite atomic Hardy space,
and the weighted variable molecular Hardy space,
which are defined, respectively, as in \cite[p.\,784]{yyy20},
\cite[Definition 1.9]{yyy20}, and \cite[pp.\,25-26]{shyy17}
with $X:=L_\omega^{p(\cdot)}({{\RR}^n})$ therein.
Combining both Lemma \ref{fsvector} and
\cite[Theorems 3.1 and 3.21]{shyy17} with
$X:=L_\omega^{p(\cdot)}({{\RR}^n})$ therein,
we obtain the following maximal function characterizations
of $H_\omega^{p(\cdot)}({{\RR}^n})$.

\begin{theorem}\label{maximal}
Let $p(\cdot)\in\cp(\rn)$, $\omega\in\mathcal{W}_{p(\cdot)}(\rn)$,
$N\in\nn\cap[\lfloor ns_{\omega}\rfloor+2,\infty)$,
and $\varphi \in \cs(\rn)$ with $\int_{\rn}\varphi(x)\,dx\neq0$.
Then the following statements are mutually equivalent:
\begin{enumerate}
\item [\rm (i)]
$f\in H_\omega^{p(\cdot)}({{\RR}^n})$;

\item[\rm (ii)]
$f\in\cs'(\rn)$ and
$\|\sup_{t\in(0,\fz)}|f\ast\varphi_t|\,\|
_{L_\omega^{p(\cdot)}({{\RR}^n})}<\infty$,
where $\varphi_t(\cdot):=\frac{1}{t^n}\varphi(\frac{\cdot}{t})$;

\item[\rm (iii)]
$f\in\cs'(\rn)$ and
$\|M_\varphi(f)\|_{L_\omega^{p(\cdot)}({{\RR}^n})}<\infty$,
here and thereafter, for any $x\in\rn$, $$M_\varphi(f)(x):=\sup_{t\in(0,\fz),\,|y-x|<t}\lf|f\ast\varphi_t(y)\r|;$$

\item[\rm (iv)]
$f\in\cs'(\rn)$ and
$\|\sup_{\varphi\in\cs_N(\rn)}M_\varphi(f)\|
_{L_\omega^{p(\cdot)}({{\RR}^n})}<\infty$,
here and thereafter,
\begin{equation*}
\cs_N(\rn):=\lf\{\varphi\in\cs(\rn):\
\sup_{\alpha\in\zz_+^n,\,|\alpha|\leq N}
\sup_{x\in\rn}(1+|x|)^N|D^\alpha\varphi(x)|\leq 1\r\}
\end{equation*}
and, for any ${\alpha}:=({\az}_1,\dots,{\az}_n)\in\zz_+^n$,
$|{\az}|:={\az}_1+\cdots+{\az}_n$ and
$D^{\az}:=(\frac\partial{\partial x_1})^{{\az}_1}
\cdots(\frac\partial{\partial x_n})^{{\az}_n}$.
\end{enumerate}
Moreover,
\begin{equation*}
\|f\|_{H_\omega^{p(\cdot)}({{\RR}^n})}
\sim\lf\|\sup_{t\in(0,\fz)}|f\ast\varphi_t|\r\|
_{L_\omega^{p(\cdot)}({{\RR}^n})}
\sim\lf\|M_\varphi(f)\r\|_{L_\omega^{p(\cdot)}({{\RR}^n})}
\sim\lf\|\sup_{\varphi\in\cs_N(\rn)}M_\varphi(f)\r\|
_{L_\omega^{p(\cdot)}({{\RR}^n})},
\end{equation*}
where implicit positive equivalence constants are independent of $f$.
\end{theorem}

\begin{remark}\label{maxrem}
We point out that the equivalence between (i) and (iv) of
Theorem \ref{maximal} coincides with \cite[Theorem 6.1]{ho17};
the rest of Theorem \ref{maximal} is completely new.
\end{remark}

By Theorem \ref{condition} and \cite[Theorem 5.1]{yyy21}
(or \cite[Theorems 3.6 and 3.7]{shyy17})
with $X:=L_\omega^{p(\cdot)}({{\RR}^n})$ therein,
we obtain the atomic characterization of
$H_\omega^{p(\cdot)}({{\RR}^n})$ as follows.

\begin{theorem}\label{atomic}
Let $p(\cdot)\in\cp(\rn)$, $\omega\in\mathcal{W}_{p(\cdot)}(\rn)$
with $s\in(1,\fz)$, $r\in(\max\{(\kappa_{\omega}^s)'/s,1\},\infty]$,
and $d\in\nn\cap[\lfloor ns-n\rfloor,\infty)$.
Then $H_\omega^{p(\cdot)}({{\RR}^n})
=H_{\omega,{\rm atom}}^{p(\cdot),r,d}(\rn)$
with equivalent quasi-norms.
\end{theorem}

\begin{remark}\label{atomicrem}
We point out that Theorem \ref{atomic} coincides with
\cite[Theorems 5.2 and 5.3]{ho17}.
\end{remark}

Throughout this article, the symbol $\mathcal{C}({{\RR}^n})$ is defined
to be the set of \emph{all continuous complex-valued functions on ${{\RR}^n}$}.
The following finite atomic characterization of $H_\omega^{p(\cdot)}({{\RR}^n})$
is a direct corollary of both Theorem \ref{condition}
and \cite[Theorem 5.2]{yyy21} (or \cite[Theorem 1.10]{yyy20})
with $X:=L_\omega^{p(\cdot)}({{\RR}^n})$ therein.

\begin{theorem}\label{finatom}
Let $p(\cdot)\in\cp(\rn)$, $\omega\in\mathcal{W}_{p(\cdot)}(\rn)$
with $s\in(1,\fz)$, $r\in(\max\{(\kappa_{\omega}^s)'/s,1\},\infty]$,
and $d\in\nn\cap[\lfloor ns-n\rfloor,\infty)$.
\begin{itemize}
\item [{\rm (i)}]
If $r\in(\max\{(\kappa_{\omega}^s)'/s,1\},\infty)$, then
$\|\cdot\|_{H_{\omega,{\rm fin}}^{p(\cdot),r,d}(\rn)}$ and
$\|\cdot\|_{H_\omega^{p(\cdot)}({{\RR}^n})}$ are equivalent quasi-norms
on $H_{\omega,{\rm fin}}^{p(\cdot),r,d}(\rn)$;

\item [{\rm (ii)}]
$\|\cdot\|_{H_{\omega,{\rm fin}}^{p(\cdot),\fz,d}(\rn)}$ and
$\|\cdot\|_{H_\omega^{p(\cdot)}({{\RR}^n})}$ are equivalent quasi-norms
on $H_{\omega,{\rm fin}}^{p(\cdot),\fz,d}(\rn)\cap \mathcal{C}(\rn)$.
\end{itemize}
\end{theorem}

\begin{remark}\label{finatomrem}
We point out that Theorem \ref{finatom} is completely new.
\end{remark}

The following molecular characterization of
$H_\omega^{p(\cdot)}({{\RR}^n})$ is a direct
corollary of both Theorem \ref{condition}
and \cite[Theorem 3.9]{shyy17} with
$X:=L_\omega^{p(\cdot)}({{\RR}^n})$ therein.

\begin{theorem}\label{molecular}
Let $p(\cdot)\in\cp(\rn)$, $\omega\in\mathcal{W}_{p(\cdot)}(\rn)$
with $s\in(1,\fz)$, $r\in(\max\{(\kappa_{\omega}^s)'/s,1\},\infty]$,
$d\in\nn\cap[\lfloor ns-n\rfloor,\infty)$, and
$$\ez\in\lf(n\lf(s-\frac{1}{\max\{(\kappa_{\omega}^s)'/s,1\}}\r),\infty\r).$$
Then $H_\omega^{p(\cdot)}({{\RR}^n})
=H_{\omega,{\rm mol}}^{p(\cdot),r,d,\ez}(\rn)$
with equivalent quasi-norms.
\end{theorem}

\begin{remark}\label{molecularrem}
We point out that Theorem \ref{molecular} is completely new.
\end{remark}

Let ${\varphi}\in{\mathcal S}({{\RR}^n})$.
For any $f\in\mathcal{S}'(\mathbb R^n)$,
the \emph{Littlewood--Paley $g$-function} $g(f)$ and
\emph{Littlewood--Paley $g_{\lambda}^*$-function}
$g_{\lambda}^*(f)$ of $f$ with ${\lambda}\in(0,{\infty})$ are defined,
respectively, by setting, for any $x\in{{\RR}^n}$,
\begin{align*}
g(f)(x):=\left[\int_0^\infty
\lf|{\varphi}(tD)(f)(x)\r|^2\,\frac{\,dt}{t}\right]^{1/2}
\end{align*}
and
\begin{align*}
g_{\lambda}^*(f)(x):=\left[\int_0^\infty\int_{\mathbb  R^n}
\left(\frac{t}{t+|x-y|}\right)^{{\lambda}n}
\left|{\varphi}(tD)(f)(y)\right|^2\,\frac{\,dy\,dt}{t^{n+1}}\right]^{1/2}.
\end{align*}
Recall that $f\in\mathcal{S}'(\mathbb R^n)$ is said to
\emph{vanish weakly at infinity} if, for any
$\varphi\in\mathcal{S}(\mathbb R^n)$,
$f*\varphi_t\to 0$ in $\mathcal{S}'(\mathbb R^n)$ as $t\to \infty$
(see, for instance, \cite[p.\,50]{fs82}).

By Theorem \ref{condition} and \cite[Theorem 4.11]{cwyz19}
with $X:=L_\omega^{p(\cdot)}({{\RR}^n})$ therein,
we obtain the Littlewood--Paley function characterizations
of $H_\omega^{p(\cdot)}({{\RR}^n})$ as follows.

\begin{theorem}\label{gfunction}
Let $p(\cdot)\in\cp(\rn)$, $\omega\in\mathcal{W}_{p(\cdot)}(\rn)$
with $s\in(1,\fz)$, ${\varphi}\in{\mathcal S}({{\RR}^n})$ satisfy
$${\mathbf{1}}_{B(\vec{0}_n,4)\backslash B(\vec{0}_n,2)}\le{\varphi}
\le{\mathbf{1}}_{B(\vec{0}_n,8)\backslash B(\vec{0}_n,1)},$$
and $\lz\in(2s,\fz)$.
Then the following statements are mutually equivalent:
\begin{enumerate}
\item [\rm (i)]
$f\in H_\omega^{p(\cdot)}({{\RR}^n})$;

\item[\rm (ii)]
$f\in\cs'(\rn)$,
$f$ vanishes weakly at infinity,
and $g(f)\in L_\omega^{p(\cdot)}({{\RR}^n})$;

\item[\rm (iii)]
$f\in\cs'(\rn)$,
$f$ vanishes weakly at infinity,
and $g_{\lambda}^*(f)\in L_\omega^{p(\cdot)}({{\RR}^n})$.
\end{enumerate}
Moreover, there exist two positive constants $C_1$ and $C_2$,
independent of $f$, such that
\begin{equation*}
\|f\|_{H_\omega^{p(\cdot)}({{\RR}^n})}
\le C_1\lf\|g(f)\r\|_{L_\omega^{p(\cdot)}({{\RR}^n})}
\le C_1\lf\|g_{\lambda}^*(f)\r\|_{L_\omega^{p(\cdot)}({{\RR}^n})}
\le C_2\lf\|f\r\|_{H_\omega^{p(\cdot)}({{\RR}^n})}.
\end{equation*}
\end{theorem}

\begin{remark}\label{gfunctionrem}
\begin{enumerate}
\item [\rm (i)]
We point out that Theorem \ref{gfunction} is completely new.

\item[\rm (ii)]
If $p(\cdot)\equiv p\in(0,1]$ and $\omega\equiv1$, then,
by Remark \ref{conrem} and Theorem \ref{gfunction},
we obtain the Littlewood--Paley $g_{\lambda}^*$ function
of $L^{p}(\mathbb{R}^n)$ with $\lz\in(2/p,\fz)$,
which coincides with the best known range.
\end{enumerate}
\end{remark}

The following Campanato-type function space is inspired by
\cite[Definition 3.2]{zhyy22}.

\begin{definition}\label{campanato}
Let $p(\cdot)\in\cp(\rn)$, $\omega\in\mathcal{W}_{p(\cdot)}(\rn)$
with $s\in(1,\fz)$, $q\in[1,{\infty})$, and $d\in\zz_+$.
Then the \emph{Campanato space}
$\mathcal{L}_{\omega,q,d}^{p(\cdot)}({{\RR}^n})$ is defined
to be the set of all $f\in L^q_{\rm loc}({{\RR}^n})$ such that
\begin{align*}
\|f\|_{\mathcal{L}_{\omega,q,d}^{p(\cdot)}({{\RR}^n})}
:=&\,\sup\lf\|\lf\{\sum_{j=1}^m
\lf[\frac{{\lambda}_j}{\|{\mathbf{1}}_{B_j}\|
_{L_\omega^{p(\cdot)}({{\RR}^n})}}\r]^{\frac{1}{s}}
{\mathbf{1}}_{B_j}\r\}^{s}\r\|
_{L_\omega^{p(\cdot)}({{\RR}^n})}^{-1}\\
&\quad\quad\times\sum_{k=1}^m\lf\{\frac{{\lambda}_k|B_k|}
{\|{\mathbf{1}}_{B_k}\|_{L_\omega^{p(\cdot)}({{\RR}^n})}}
\lf[\frac1{|B_k|}\int_{B_k}
\lf|f(x)-P^d_{B_k}f(x)\r|^q \,dx\r]^\frac1q\r\}<\fz,
\end{align*}
where $P^d_Bf$ denotes the minimizing polynomial of
$f$ on $B$ with degree not greater than $d$,
and the supremum is taken over all $m\in\nn$,
balls $\{B_j\}^m_{j=1}$ in ${\RR}^n$,
and $\{\lz_j\}^m_{j=1}\st[0,\fz)$
with $\sum_{j=1}^m\lz_j\neq0$.
\end{definition}

By Theorem \ref{condition}, \cite[p.\,73]{cfbook}
(or \cite[Lemma 2.3.16]{dhr11}), and \cite[Theorem 3.14]{zhyy22}
with $X:=L_\omega^{p(\cdot)}({{\RR}^n})$ therein,
we obtain the duality of $H_\omega^{p(\cdot)}({{\RR}^n})$ as follows.

\begin{theorem}\label{dual}
Let $p(\cdot)\in\cp(\rn)$, $\omega\in\mathcal{W}_{p(\cdot)}(\rn)$
with $s\in(1,\fz)$, $r\in(\max\{(\kappa_{\omega}^s)'/s,1\},\infty]$,
and $d\in\nn\cap[\lfloor ns-n\rfloor,\infty)$.
Then the dual space of $H_\omega^{p(\cdot)}({{\RR}^n})$,
denoted by $(H_\omega^{p(\cdot)}({{\RR}^n}))^*$,
is $\mathcal{L}_{\omega,r',d}^{p(\cdot)}({{\RR}^n})$
with $1/r+1/r'=1$ in the following sense:
\begin{enumerate}
\item[{\rm (i)}] Let $f\in\mathcal{L}_{\omega,r',d}^{p(\cdot)}({{\RR}^n})$.
Then the linear functional
\begin{align}\label{10.21.x1}
T_f:\ g\rightarrow T_f(g):=\int_{{{\RR}^n}}f(x)g(x)\,dx,
\end{align}
originally defined for any $g\in H_{\omega,{\rm fin}}^{p(\cdot),r,d}(\rn)$,
has a bounded extension to $H_\omega^{p(\cdot)}({{\RR}^n})$.

\item[{\rm (ii)}] Conversely, every continuous linear
functional on $H_\omega^{p(\cdot)}({{\RR}^n})$ arises
as in \eqref{10.21.x1} with a unique
$f\in\mathcal{L}_{\omega,r',d}^{p(\cdot)}({{\RR}^n})$.
\end{enumerate}

Moreover,
$\|f\|_{\mathcal{L}_{\omega,r',d}^{p(\cdot)}({{\RR}^n})}
\sim\|T_f\|_{(H_\omega^{p(\cdot)}({{\RR}^n}))^*}$,
where the implicit positive equivalence constant
is independent of $f$.
\end{theorem}

\begin{remark}\label{dualrem}
We point out that Theorem \ref{dual} is completely new.
\end{remark}

Throughout this article, for any
$\delta:=({\delta}_1,\ldots,{\delta}_n)\in\zz^n_+$ and
$x:=(x_1,\ldots,x_n)\in{{\RR}^n}$,
$$x^{{\delta}}:=x_1^{{\delta}_1}\cdots x_n^{{\delta}_n}.$$
For any $d\in\zz_+$, we use ${\mathcal C}^d({{\RR}^n})$ to denote the
\emph{set of all functions having continuous classical derivatives up to order  $d$}.
For any given $\alpha\in(0,1]$ and $d\in\zz_+$,
let ${\mathcal C}_{\alpha,d}({{\RR}^n})$ be the
\emph{set of all functions} $\varphi\in{\mathcal C}^d({{\RR}^n})$ such that
$\mathop\mathrm{\,supp\,}\varphi:=\overline{\{x\in\rn:\ \varphi(x)\neq0\}}
\subset\{x\in\mathbb R^n:\ |x|\le 1\}$,
$\int_{\mathbb R^n}\varphi(x)x^{\gamma}\,dx=0$ for any ${\gamma}\in\zz_+^n$
with $|{\gamma}|\le d$, and, for any $x_1,x_2\in{{\RR}^n}$ and
$\nu\in\zz_+^n$ with $|\nu|=d$,
\begin{align*}
|D^\nu\varphi(x_1)-D^\nu\varphi(x_2)|\le|x_1-x_2|^\alpha.
\end{align*}
For any $f\in L_{\rm loc}^1({{\RR}^n})$
and $(y,t)\in{\mathbb{R}}_+^{n+1}$, let
\begin{align*}
A_{\alpha,d}(f)(y,t):=
\sup_{\varphi\in{\mathcal C}_{\alpha,d}({{\RR}^n})}|f*\varphi_t(y)|.
\end{align*}
Then the \emph{intrinsic Littlewood--Paley $g$-function} $g_{\alpha,d}(f)$,
the \emph{intrinsic Littlewood--Paley $g_\lambda^*$-function}
$g_{\lambda,\alpha,d}^*(f)$ with $\lambda\in(0,{\infty})$,
and the \emph{intrinsic Lusin area function}
$S_{\alpha,d}(f)$ of $f$ are defined, respectively, by setting,
for any $x\in{{\RR}^n}$,
\begin{align*}
g_{\alpha,d}(f)(x):=
\left\{\int_0^\infty\left[A_{\alpha,d}(f)(x,t)\right]^2\,\frac{dt}{t}\right\}^{1/2},
\end{align*}
\begin{align*}
g_{\lambda,\alpha,d}^*(f)(x):=\left\{\int_0^\infty\int_{\mathbb  R^n}
\left(\frac{t}{t+|x-y|}\right)^{{\lambda} n}\left[A_{\alpha,d}(f)(y,t)\right]^2\,
\frac{\,dy\,dt}{t^{n+1}}\right\}^{1/2},
\end{align*}
and
\begin{align*}
S_{\alpha,d}(f)(x):=\left\{\int_{\Gamma(x)}\left[A_{\alpha,d}(f)(y,t)
\right]^2\,\frac{\,dy\,dt}{t^{n+1}}\right\}^{1/2}.
\end{align*}

By Theorem \ref{condition}, Theorem \ref{dual},
and \cite[Theorems 5.5 and 5.6]{yyy21}
(or \cite[Theorems 1.15 and 1.16]{yyy20}) with
$X:=L_\omega^{p(\cdot)}({{\RR}^n})$ therein,
we obtain the intrinsic square function characterizations
of $H_\omega^{p(\cdot)}({{\RR}^n})$ as follows.

\begin{theorem}\label{intfunction}
Let $\alpha\in(0,1]$ and $d\in\zz_+$. Assume that $p(\cdot)\in\cp(\rn)$, $\omega\in\mathcal{W}_{p(\cdot)}(\rn)$ with $s\in(1,(n+\az+d)/n)$,
and
$$\lz\in\lf(\max\lf\{2s,2s+1-\frac{2}{\max\{(\kappa_{\omega}^s)'/s,1\}}\r\},\fz\r).$$
Then the following statements are mutually equivalent:
\begin{enumerate}
\item [\rm (i)]
$f\in H_\omega^{p(\cdot)}({{\RR}^n})$;

\item[\rm (ii)]
$f\in(\mathcal{L}_{\omega,1,d}^{p(\cdot)}({{\RR}^n}))^*$,
$f$ vanishes weakly at infinity,
and $g_{\alpha,d}(f)\in L_\omega^{p(\cdot)}({{\RR}^n})$;

\item[\rm (iii)]
$f\in(\mathcal{L}_{\omega,1,d}^{p(\cdot)}({{\RR}^n}))^*$,
$f$ vanishes weakly at infinity,
and $g_{\lambda,\alpha,d}^*(f)\in L_\omega^{p(\cdot)}({{\RR}^n})$;

\item[\rm (iv)]
$f\in(\mathcal{L}_{\omega,1,d}^{p(\cdot)}({{\RR}^n}))^*$,
$f$ vanishes weakly at infinity,
and $S_{\alpha,d}(f)\in L_\omega^{p(\cdot)}({{\RR}^n})$.
\end{enumerate}
Moreover,
\begin{equation*}
\|f\|_{H_\omega^{p(\cdot)}({{\RR}^n})}
\sim\lf\|g_{\alpha,d}(f)\r\|_{L_\omega^{p(\cdot)}({{\RR}^n})}
\sim\lf\|g_{\lambda,\alpha,d}^*(f)\r\|_{L_\omega^{p(\cdot)}({{\RR}^n})}
\sim\lf\|S_{\alpha,d}(f)\r\|_{L_\omega^{p(\cdot)}({{\RR}^n})},
\end{equation*}
where implicit positive equivalence constants are independent of $f$.
\end{theorem}

\begin{remark}\label{intfunctionrem}
\begin{enumerate}
\item [\rm (i)]
We point out that, when $d=0$, the boundedness of intrinsic
square functions in Theorem \ref{intfunction} improves the
corresponding results in \cite[Theorems 4.8 and 4.9]{ho19} by
weakening the condition $\omega^{p_*}\in\mathcal{A}_{p(\cdot)/p_*}(\rn)$
into $\omega\in\mathcal{W}_{p(\cdot)}(\rn)$;
the rest of Theorem \ref{intfunction} is completely new.

\item[\rm (ii)]
If $p(\cdot)\equiv p\in(0,1]$ and $\omega\equiv1$, then,
by Remark \ref{conrem} and Theorem \ref{intfunction},
we obtain the intrinsic Littlewood--Paley $g_{\lambda}^*$
function of $L^{p}(\mathbb{R}^n)$ with $\lz\in(2/p,\fz)$,
which coincides with the best known range.
\end{enumerate}
\end{remark}

\section{Boundedness of Some Sublinear Operators
on Weighted Variable Hardy Spaces\label{s4}}

In this section, we obtain the boundedness of some sublinear operators,
including Calder\'on--Zygmund operators and Bochner--Riesz means,
on weighted variable Hardy spaces $H_\omega^{p(\cdot)}({{\RR}^n})$
by viewing $H_\omega^{p(\cdot)}({{\RR}^n})$ as special cases
of Hardy spaces $H_X({{\RR}^n})$ associated with ball
quasi-Banach function spaces $X$ and applying known results
of $H_X({{\RR}^n})$ obtained in \cite{tz23,wyy}.

\subsection{Calder\'on--Zygmund Operators\label{s4x}}

Recall that, for any given $\delta\in(0,1)$,
a \emph{convolutional $\delta$-type Calder\'on--Zygmund operator}
$T$ means that: $T$ is a linear bounded operator on $L^2(\rn)$
with kernel $k\in\cs'(\rn)$ coinciding with a locally integrable
function on $\rn\backslash \{\vec{0}_n\}$ and satisfying that
there exists a positive constant $C$ such that,
for any $x$, $y\in\rn$ with $|x|>2|y|$,
$$|k(x-y)-k(x)|\le C\frac{|y|^\delta}{|x|^{n+\delta}},$$
and that, for any $f\in L^2(\rn)$, $Tf(x)=k*f(x)$.

Throughout this article,
for any $\beta=(\beta_1,\ldots,\beta_n)\in \zz_+^n$,
any $\beta$-order differentiable function $F(\cdot,\cdot)$
on $\rn\times \rn$, and any $(x,y)\in \rn\times \rn$, let
$$\partial_{(2)}^{\beta}F(x,y):=\frac{\partial^{|\beta|}}
{\partial y_1^{\beta_1}\cdots\partial y_n^{\beta_n}}F(x,y).$$
Recall that, for any given $\gamma\in(0,\fz)$, a linear operator
$T$ is called a \emph{$\gamma$-order Calder\'on--Zygmund operator}
if $T$ is bounded on $L^2(\rn)$ and its kernel
$$k(x,y):\ (\rn\times\rn)\backslash\{(x,x):\ x\in\rn\}\to\mathbb C$$
satisfies that there exists a positive constant $C$ such that,
for any $\az\in\zz^n_+$ with $|\az|=\lceil\gamma\rceil-1$
and any $x,y,z\in\rn$ with $|x-y|>2|y-z|$,
\begin{align*}
\lf|\partial_{(2)}^\az k(x,y)-\partial_{(2)}^\az k(x,z)\r|\le
C\frac{|y-z|^{\gamma-\lceil\gamma\rceil+1}}{|x-y|^{n+\gamma}},
\end{align*}
and that, for any $f\in L^2(\rn)$ having compact support and $x\notin\supp f$,
$$Tf(x)=\int_{\supp f}k(x,y)f(y)\,dy.$$
For any given $m\in\nn$, an operator $T$ is said to have
the \emph{vanishing moment condition up to order $m$} if,
for any $a\in L^2(\rn)$ with compact support satisfying that,
for any $\beta\in\zz^n_+$ with $|\beta|\le m$, $\int_{\rn}a(x)x^\beta\,dx=0$,
it holds true that $\int_{\rn}Ta(x)x^\beta\,dx=0$.

By Theorem \ref{condition}, \cite[p.\,73]{cfbook}
(or \cite[Lemma 2.3.16]{dhr11}), and \cite[Theorems 3.5 and 3.11]{wyy}
with $X:=L_\omega^{p(\cdot)}({{\RR}^n})$ therein,
we obtain the boundedness of Calder\'on--Zygmund operators
on $H_\omega^{p(\cdot)}({{\RR}^n})$ as follows.

\begin{theorem}\label{cz-bdn}
Assume that $p(\cdot)\in\cp(\rn)$ and
$\omega\in\mathcal{W}_{p(\cdot)}(\rn)$ with $s\in(1,\fz)$.
\begin{enumerate}
\item [\rm (i)]
Let $\delta\in(0,1)$ and $T$ be a convolutional
$\delta$-type Calder\'on--Zygmund operator.
If $s\in(1,\frac{n+\delta}{n})$, then $T$ has an unique
extension on $H_\omega^{p(\cdot)}({{\RR}^n})$ and, moreover,
for any $f\in H_\omega^{p(\cdot)}({{\RR}^n})$,
$$\|Tf\|_{H_\omega^{p(\cdot)}({{\RR}^n})}\le
C\|f\|_{H_\omega^{p(\cdot)}({{\RR}^n})},$$
where $C$ is a positive constant independent of $f$.

\item[\rm (ii)]
Let $\gamma\in(0,\fz)$ and $T$ be a $\gamma$-order
Calder\'on--Zygmund operator with the vanishing
moment condition up to order $\lceil\gamma\rceil-1$.
If $s\in(\frac{n+\lceil\gamma\rceil-1}{n},\frac{n+\gamma}{n})$,
then $T$ has a unique extension on $H_\omega^{p(\cdot)}({{\RR}^n})$ and,
moreover, for any $f\in H_\omega^{p(\cdot)}({{\RR}^n})$,
$$\|Tf\|_{H_\omega^{p(\cdot)}({{\RR}^n})}\le
C\|f\|_{H_\omega^{p(\cdot)}({{\RR}^n})},$$
where $C$ is a positive constant independent of $f$.
\end{enumerate}
\end{theorem}

\begin{remark}\label{cz-bdnrem}
\begin{enumerate}
\item [\rm (i)]
We point out that a different variant of Theorem \ref{cz-bdn}(ii)
was obtained in \cite[Theorem 4.3]{ho19} which uses different
Calder\'on--Zygmund operators and the condition
$\omega^{p_*}\in\mathcal{A}_{p(\cdot)/p_*}(\rn)$
instead of $\omega\in\mathcal{W}_{p(\cdot)}(\rn)$.

\item[\rm (ii)]
If $p(\cdot)\equiv p\in(0,1]$ and $\omega\equiv1$, then,
by Remark \ref{conrem} and Theorem \ref{cz-bdn},
we obtain the boundedness of convolutional $\delta$-type
and non-convolutional $\gamma$-order Calder\'on--Zygmund operators,
with the vanishing moment condition up to order $\lceil\gamma\rceil-1$,
on $H^{p}({{\RR}^n})$ with, respectively, $p\in(\frac{n}{n+\delta},1]$
and $p\in(\frac{n}{n+\gamma},\frac{n}{n+\lceil\gamma\rceil-1}]$.
\end{enumerate}
\end{remark}

\subsection{Bochner--Riesz Means\label{s4y}}
We first recall the definitions of the Bochner--Riesz means.
Let $\delta,\epsilon\in(0,\fz)$. For any given
$f\in{\mathcal S}({{\RR}^n})$, the Bochner--Riesz operator
$B^{\delta}_{1/\epsilon}(f)$ of order $\delta$ is defined by setting,
for any $x\in\rn$,
$$B^{\delta}_{1/\epsilon}(f)(x):=(f*\varphi_{\epsilon})(x),$$
where $\varphi(x):=\mathscr{F}((1-|\cdot|^2)_+^{\delta})(x)$
and, for any $a\in\rn$, $(a)_+:=\max\{a,0\}$;
moreover, the maximal Bochner--Riesz means $B^{\delta}_{*}(f)$ is
defined by setting, for any $x\in\rn$,
$$B^{\delta}_{*}(f)(x)
:=\sup_{\epsilon\in(0,\fz)}\lf|B^{\delta}_{1/\epsilon}(f)(x)\r|.$$

By Theorem \ref{condition}, \cite[p.\,73]{cfbook}
(or \cite[Lemma 2.3.16]{dhr11}), and \cite[Theorem 1.8]{tz23}
with $X:=L_\omega^{p(\cdot)}({{\RR}^n})$ therein,
we obtain the boundedness of the maximal Bochner--Riesz means
on $H_\omega^{p(\cdot)}({{\RR}^n})$ as follows.

\begin{theorem}\label{br-bdn}
Let $p(\cdot)\in\cp(\rn)$,
$\omega\in\mathcal{W}_{p(\cdot)}(\rn)$ with $s\in(1,\fz)$,
$\delta\in(\frac{n-1}{2},\fz)$, and $B^{\delta}_{*}$ be the
maximal Bochner--Riesz means.
If $s\in(1,\frac{n+1+2\delta}{2n})$, then $B^{\delta}_{*}$ has
an unique extension on $H_\omega^{p(\cdot)}({{\RR}^n})$ and,
moreover, for any $f\in H_\omega^{p(\cdot)}({{\RR}^n})$,
$$\lf\|B^{\delta}_{*}(f)\r\|_{L_\omega^{p(\cdot)}({{\RR}^n})}\le
C\|f\|_{H_\omega^{p(\cdot)}({{\RR}^n})},$$
where $C$ is a positive constant independent of $f$.
\end{theorem}

\begin{remark}\label{br-bdnrem}
\begin{enumerate}
\item [\rm (i)]
We point out that Theorem \ref{br-bdn} improves the corresponding
results in \cite[Theorem 4.5]{ho19} by weakening the condition
$\omega^{p_*}\in\mathcal{A}_{p(\cdot)/p_*}(\rn)$
into $\omega\in\mathcal{W}_{p(\cdot)}(\rn)$.

\item[\rm (ii)]
If $p(\cdot)\equiv p\in(0,1]$ and $\omega\equiv1$, then,
by Remark \ref{conrem} and Theorem \ref{br-bdn},
we obtain the boundedness of maximal Bochner--Riesz means
$B^{\delta}_{*}$ on $H^{p}({{\RR}^n})$ with
$p\in(\frac{2n}{n+1+2\delta},1]$.
\end{enumerate}
\end{remark}

At the end of this section, we point out that,
since the main results of this article do not depend
on Lebesgue measure and Euclidean metric of $\rn$,
via viewing weighted variable Lebesgue spaces as special cases
of ball quasi-Banach function spaces and applying known results
of in \cite{lyy23,lyy24,wyy24,yhyy1,yhyy2}, we can obtain some similar
real-variable results of $H_\omega^{p(\cdot)}$ on anisotropic Euclidean
spaces or spaces of homogeneous type; we omit the details.

\medskip

\noindent\textbf{Acknowledgements}\quad  This project is supported by
the National Natural Science Foundation of China (Grant No. 12301112),
China Postdoctoral Science Foundation (Grant Nos. 2022M721024 and
2023M741020), Postdoctoral Fellowship Program of CPSF (Grant No. GZC20230695),
and the Open Project Program of Key Laboratory of Mathematics and
Complex System of Beijing Normal University (Grant No. K202304).

\medskip

\noindent\textbf{Author Contributions}\quad All authors
developed and discussed the results and contributed to the final manuscript.

\medskip

\noindent\textbf{Data Availability Statement}\quad Data sharing
is not applicable to this article as no data sets were generated or
analysed.

\medskip

\noindent\textbf{Declarations}

\medskip

\noindent\textbf{Conflict of interest}\quad All authors state no conflict of interest.

\medskip

\noindent\textbf{Informed consent}\quad Informed consent has been obtained from all individuals included in this research work.

\end{document}